\documentclass[11pt,a4paper,english]{article}
\usepackage{babel}
\usepackage{subfigure}
\usepackage{longtable}
\usepackage{tikz}
\usepackage[T1]{fontenc}
\usepackage[utf8]{inputenc}
\usepackage{hyperref}
\usepackage{setspace}
\usepackage{graphicx}
\usepackage{float} 
\usepackage{blkarray}
\usepackage{amssymb}
\usepackage{amsmath, amsthm}
\usepackage{mathtools}
\usepackage{mathrsfs}
\usepackage{dsfont}

\usepackage{enumitem}
 \usepackage{tabulary}
 \usepackage{subfigure}
\usepackage{centernot}

\newtheorem{lemma}{Lemma}[section]
\newtheorem*{theorem*}{Theorem}
\newtheorem{theorem}[lemma]{Theorem}
\newtheorem{proposition}[lemma]{Proposition}
\newtheorem{definition}[lemma]{Definition}

\newtheorem{remark}[lemma]{Remark}
\newtheorem{corollary}{Corollary}
\usepackage[hmargin=3.0cm,vmargin=3.0cm]{geometry}

\usepackage{dsfont}

\newcommand{\p}{\Bbb{P}}

\newcommand{\N}{\mathbb{N}}
\newcommand{\Z}{\mathbb{Z}}
\newcommand{\R}{\mathbb{R}}

\newcommand{\E}{\ensuremath{\mathbb{E}}}

\begin{document}
\title{The Symmetric Coalescent and Wright-Fisher models with bottlenecks}

 \author{Adri\'an Gonz\'alez Casanova, {Ver\'onica Mir\'o Pina}, Arno Siri-J{\'e}gousse\\
 \emph{Corresponding author:}  \href{mailto:veronica.miropina@normalesup.org}{veronica.miropina@normalesup.org}}




\maketitle

\begin{abstract}
We define a new class of $\Xi$-coalescents characterized by a possibly infinite measure over the non negative integers.
We call them symmetric coalescents since they are the unique family of exchangeable coalescents satisfying a symmetry property on their coagulation rates: they are invariant under any transformation that  consists of moving one element from one block to another without changing the total number of blocks.
We illustrate the diversity of behaviors of this family of processes by introducing and studying a one parameter  subclass, the $(\beta,S)$-coalescents.
We also embed this family in a larger class of $\Xi$-coalescents arising as the limit genealogies of Wright-Fisher models with bottlenecks.
Some convergence results rely on a new Skorokhod type metric, that induces the Meyer-Zheng topology, which allows us to study the scaling limit of non-Markovian processes using standard techniques.
\end{abstract}

\section{Introduction}

\subsection{Wright-Fisher models with demographic bottlenecks}
\label{WFbottlenecks}
Since it was proposed in 1982, the Kingman coalescent \cite{Kingman} has become a key tool in population genetics. 
It can describe the limit genealogy of classical models such as the Wright-Fisher and the Moran model.
It has proven to be robust to modifications of these models' assumptions (such as constant population size or random mating) and thus arises as the genealogy of a
broad class of population models.
However, it does not model well genealogies from certain populations, e.g. with a skewed offspring distribution, which are captured by coalescents with (simultaneous) multiple collisions \cite{mohle3}. 

Modeling populations with varying population size has been of great interest, for example to infer the human population history \cite{LiDurbin, Terhorst}.
Several variations of the Wright-Fisher model with fluctuating  population size have been studied, in which the population size changes but remains of order $N$. It has been shown that in many cases, the genealogy converges to a continuous time-rescaling of the Kingman coalescent (see for example \cite{Tavare,Kaj,Jagers}). 
More recently,  Freund \cite{Freund}  studied the case of Cannings models with highly variant offspring number (whose genealogy is usually described by a $\Lambda$-coalescent) in which the population size fluctuates (but remains of the order of $N$) and has shown that the genealogy converges to a time-rescaled $\Lambda$-coalescent.

In Section 6.1 of \cite{BBMST}, Birkner et al. consider a population undergoing recurrent {\it demographic bottlenecks}.
We call a bottleneck an event that reduces substantially the population size and that may last for one or several generations.
They suggest that the genealogy is described by a discontinuous time-rescaling of the Kingman coalescent, more precisely a Kingman coalescent where time is rescaled by a subordinator, and which is in fact a coalescent with simultaneous multiple collisions. 
But they only consider the case  where the population size during the bottleneck is small compared to  $N$ but still tends to infinity as $N \to \infty$.

To our knowledge, the case of \textit{drastic} fluctuations, in which the population size during the bottleneck does not tend to infinity as $N\to \infty$, has not been studied yet. 
In this article, we are going to study different types of bottlenecks, with different scalings for the population sizes inside and outside the bottleneck, and different lengths. We will  establish a classification of the limiting genealogies obtained in the different settings and give some intuitions on how to relate the different processes.
To do so, we define a class of models that can be called Wright-Fisher models with demographic bottlenecks. 
\begin{definition}[The Wright-Fisher model with bottlenecks]
The Wright-Fisher model with bottlenecks (parametrized by $N \in \N$) has varying population size, which is given by
  a sequence of random variables $\{R^N_g\}_{g \in \Z_+}$  taking values in $[N]=\{1, \dots, N\}$.
  It is the random graph $(V,E)$ where $V=\{(i,g):i\in [R_g^N],g\in \Z_+\}$, each individual  $(i,g)\in V$ chooses her parent uniformly amongst the $R_{g-1}^N$ individuals of generation $g-1$, and the set of edges is $E:=\{((j,g-1)(i,g)): (j,g-1)$ is the parent of $(i,g),g\in \Z_+\}$. 
  \label{def2}
\end{definition}

The case $\p(R^N_g=N)=1, \forall g \in \Z_+$, is the classical Wright-Fisher model and it is well known that, when the time is rescaled by $N$ and  $N \to \infty$, the genealogy of a sample of $n$ individuals is described by the Kingman coalescent. We are led to ask ourselves under which conditions on $\{R^N_g\}_{g \in \Z_+}$ does the genealogy still converge to a Kingman coalescent,
and if it does not, what type of coalescents describe the genealogy of a population that has undergone bottlenecks.

We are going to study different types of Wright-Fisher models with bottlenecks, with different types of laws for the sequence $\{R^N_g\}_{g \in \Z_+}$.
Inspired by \cite{BBMST} (Section 6, p. 57), we are going to describe the demographic history of the population by three random sequences of i.i.d. positive real numbers: $\{s_{i,N}\}_{i \in N}$, $\{l_{i,N}\}_{i \in N}$ and $\{b_{i,N}\}_{i \in N}$. The sequence of population sizes $\{R^N_g\}_{g \in \Z_+}$ is given by

$$
R^N_g \ = \ \left\{ 
\begin{array}{ll}
b_{m,N} N  \textrm{ if } \sum_{i=1}^{m-1}(s_{i, N} + l_{i,N}) + s_{m,N} < g  \le  \sum_{i=1}^{m}(s_{i, N} + l_{i,N}) \\
N \textrm{ otherwise.}
\end{array} 
\right.
$$
This means that the population size stays at $N$ for $s_{i,N}$ generations and then it is reduced to $b_{i,N}N$ for $l_{i,N}$ generations. At the end of the bottleneck, the population reaches $N$ again and it stays until the next event. 
Note that we have assumed that the decline and the re-growth of the population size are instantaneous.

We call $b_{i,N}$ the \textit{intensity}  of the $i$-{th} bottleneck and we distinguish between:
\begin{itemize}
\item \textit{Soft} bottlenecks, where $b_{i,N} \to 0$ in distribution but  $Nb_{i,N} \to \infty$ in distribution i.e. the population size during the bottleneck is small compared to $N$, but still large in absolute numbers. 
\item \textit{Drastic} bottlenecks, where $b_{i,N} \to 0$ in distribution and  $Nb_{i,N} < \infty$ in distribution (as $N \to \infty$) i.e. the population size during the bottleneck is very small compared to $N$, and remains finite  in the limiting scenario, when the population size outside the bottlenecks is infinite. 
\end{itemize} 
We call $l_{i,N}$ the \textit{duration} of the $i$-{th} bottleneck. We will distinguish between \textit{short} bottlenecks, that last for only one generation and \textit{long} bottlenecks that last for several generations. In both cases, we will assume that there exists $\alpha \in (0,1]$ such that $l_{i,N}N^\alpha \to 0$  as $N \to \infty$, in distribution, i.e. when time is re-scaled by $N^\alpha$ the duration of the bottleneck is negligible. 
Finally, we call $s_{i,N}$ the \textit{periodicity} of the bottlenecks and again, we distinguish between \textit{frequent} bottlenecks when, $s_{i,N}/N\to 0$ in distribution as $N \to \infty$ and \textit{rare} bottlenecks  otherwise.

As we shall see, coalescents with simultaneous multiple collisions arise as the limiting genealogies for the Wright-Fisher model with bottlenecks. In the case of short drastic bottlenecks, the genealogies are described by a new family of coalescents that we define and study.

\subsection{A new family of $\Xi$-coalescents}

Coalescents with simultaneous multiple collisions ($\Xi$-coalescents, \cite{S2000,mohle3, BertoinLeGall, Bertoin}) form the widest class of exchangeable coagulating Markov chains with values in the set of partitions of $\N$.
Mathematically, they give a nice connection with de Finetti's representation of exchangeable partitions and they exhibit a rich variety of behaviors.
Biologically, they describe the genealogy of a large class of population models and their study provides some statistical tools for inference.
Schweinsberg \cite{S2000} showed that any exchangeable coalescent is characterized by a finite measure $\Xi$ on the ranked infinite simplex $$\Delta = \{ \mathbf{\zeta} = (\zeta_1, \zeta_2. \dots), \ \zeta_1 \ge \zeta_2 \ge \dots \ge 0, \ \sum_{i = 1}^\infty \zeta_i \le 1 \}.$$
Its dynamics are described as follows. We decompose $\Xi$ into a `Kingman part' and a `simultaneous multiple collisions' part, i.e. $\Xi = a \delta_{(0, 0, \dots)} + \Xi^0$ with $a \in [0, \infty)$ and $\Xi^0(\{(0, 0, \dots)\}) = 0 $. 
A $[b, (k_1,\dots, k_r),s]$-collision is a merger of  $b$ blocks into $r$ new blocks and $s$ unchanged blocks.
Each new block contains $k_1, \dots, k_r\ge 2$ original blocks, so that $\sum_{i=1}^rk_i=b-s$. Note that the order of the $k_1, \dots, k_r$ does not matter. 
Each $[b, (k_1,\dots, k_r),s]$-collision happens at some fixed rate 
\begin{equation}\label{ratesXi}
\lambda_{b, (k_1,\dots, k_r),s}= \ a\  \mathds{1}_{\{r = 1, s = b-1\}}  + \int_\Delta\sum_{l=0}^s\binom{s}{l}(1-\sum_{i\ge1}\zeta_i)^{s-l}\sum_{i_1\neq\dots\neq i_{r+l}}\zeta_{i_i}^{k_1}\dots\zeta_{i_r}^{k_r}\zeta_{i_{r+1}}\dots\zeta_{i_{r+l}}\frac{\Xi^0(d\zeta)}{(\zeta,\zeta)}
\end{equation}
where $(\zeta,\zeta):=\sum_{i\ge1}\zeta_i^2$.
This complicated formula is resulting from colliding original blocks according to a Kingman's paintbox associated with a partition $\zeta$ drawn from the $\sigma$-finite measure $\Xi^0(d\zeta)/(\zeta,\zeta)$.
This complexity justifies the necessity to consider subclasses of exchangeable coalescents that are easier to study. 

When $\Xi$ only puts weight on  the subset of mass-partitions having only one positive element, formula \eqref{ratesXi} reduces considerably as there is now no possibility to obtain simultaneous collisions. 
The resulting subclass of exchangeable coalescents is that of coalescents with multiple collisions ($\Lambda$-coalescents, \cite{Pitman, Sagitov}).
Their elegant theory, built around their one to one correspondence with finite measures in $[0, 1]$,
turned this family into the most studied class of coalescent processes for twenty years now.
In particular, $Beta$-coalescents \cite{Sch03} provide a one-parameter family of $\Lambda$-coalescents, more convenient to study and better calibrated for statistical applications in population genetics.
This model is now validated by the biological community \cite{Eldon, AtlanticCod, Sardine}.

However, we find fewer results about $\Xi$-coalescents in the literature and their applications in biology are rarer. 
We can yet cite the Poisson-Dirichlet coalescent \cite{Sag03, Moh10} and the Beta-Xi family \cite{BBE, BLS} that provide promising models.
The reason for this is probably the difficulty arising from the complex formulation of \eqref{ratesXi}.
A first step to simplify it is to consider a measure $\Xi$ on
$$\Delta^* = \{ \mathbf{\zeta} = (\zeta_1, \zeta_2. \dots), \ \zeta_1 \ge \zeta_2 \ge \dots \ge 0, \ \sum_{i = 1}^\infty \zeta_i = 1 \}.$$
In this case the transition rates simplify.
We can now consider $[b, (k_1,\dots, k_r)]$-collisions where $b$ blocks merge into $r$ blocks,
each one containing $k_1, \dots, k_r\ge 1$ original blocks.
Each $[b, (k_1,\dots, k_r)]$-collision happens at some fixed rate 
\begin{equation}\label{ratesXistar}
\lambda_{b, (k_1,\dots, k_r)}=  \ a\  \mathds{1}_{\{r = b-1, k_1 = 2\}}  + \int_{\Delta^*}\sum_{i_1\neq\dots\neq i_{r}}\zeta_{i_i}^{k_1}\dots\zeta_{i_r}^{k_r}\frac{\Xi^0(d\zeta)}{(\zeta,\zeta)}.
\end{equation}

In this paper we describe a simple family of $\Xi$-coalescents: \textit{the symmetric coalescent}, which arises as the limiting genealogy for Wright-Fisher models with short drastic bottlenecks. 
The reason for its name is that the distribution of the tree obtained from a symmetric coalescent is invariant under the transformation that involves cutting one branch from one node and pasting it somewhere else in the tree, at the same height  (see Figure \ref{SymetricTree} for an illustration). In other words, as a partition-valued process, the symmetric coalescent is invariant under the transformation that consists of displacing one element from one block to another (without changing the number of non-empty blocks). 

\begin{figure}
\begin{center}
\includegraphics[width=15cm]{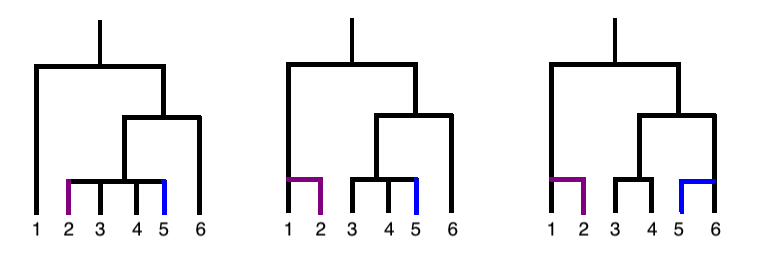}
\caption{In the symmetric coalescent these three labelled trees have the same probability. The second one is obtained from the first one by cutting and pasting the purple branch to a different node. The third one can be obtained from the second one by displacing the blue branch from one position to another.  In the $S$-coalescent, $\lambda_{6,(4,1,1)} = \lambda_{6,(2,3,1)} = \lambda_{6,(2,2,2)}$.}
\label{SymetricTree}
\end{center}
\end{figure}
 
\begin{definition}
The symmetric coalescents are the  exchangeable coalescents whose  transition rates satisfy the following symmetry property:
for every $b>1$, $2\le r <b$ and for every $k_1, \ldots, k_r$ and $k'_1, \ldots, k'_r$ such that  $\sum_{i=1}^r k_i = \sum_{i=1}^r k'_i = b$, $$\lambda_{b, (k_1, \dots, k_r)} = \lambda_{b, (k'_1, \dots, k'_r)}.$$
\label{def1}
\end{definition}

In the sequel we will consider the symmetric elements of $\Delta^*$.
Let $\xi^0:=(0,0,\dots)$ and, for $k\in\N$, $$\xi^k := (\frac{1}{k}, \dots, \frac{1}{k}, 0, \dots)$$
and we denote the set of the symmetric elements of $\Delta^*$ by $\Delta^{sym}:=\{\xi^k,k\in\N_0\}$.
Our first result establishes a correspondence between symmetric coalescents and measures on a simple set, being here $\Z_+:=\N\cup\{0\}.$

\begin{theorem}
A coalescent is symmetric if and only if there exists a measure $F$ on $\Z_+$ such that 
\begin{equation}
F(0) < \infty  \ \textrm{ and } \ \sum_{k \ge1} \frac{F(k)}{k} < \infty
\label{Ffinita}
\end{equation}
and such that its characterizing measure $S$ on $\Delta$ only puts weight on $\Delta^{sym}$ and
\begin{equation}\label{StoF}
S(\xi^k) \ = \ \left\{ 
\begin{array}{ll}
\frac{F(k)}{k} & \textrm{ if } k\in\N \\
F(0) & \textrm{ if } k=0
\end{array} 
\right..
\end{equation}
\label{propXiF}
\end{theorem}
Condition \eqref{Ffinita} ensures that the measure $S$ is finite which is a necessary and sufficient condition for a $\Xi$-coalescent to be well defined.
Observe that for $k\in\N$, $(\xi^k, \xi^k)=\frac{1}{k}$, so the rate of a $k$ merger is $F(k)$. 
Mimicking the common notations we will speak about {\it $S$-coalescents}.

Before going into further detail, we start by recalling a useful tool, which is Kingman's paintbox construction of $\Xi$-coalescents (\cite{Kingman}). We will only discuss the case when $\Xi$ is concentrated on $\Delta^*$. 
Each element in $\Delta^*$ can be seen as a tiling of (0,1), where the sizes of the subintervals are $\zeta_1, \zeta_2, \dots$.
The $\Xi$-coalescent can be constructed as follows: when there are $b$ blocks, for every $\zeta \in \Delta^*$, at rate $\Xi(d\zeta)/(\zeta,\zeta)$, we choose the tiling associated with $\zeta$, then we throw $b$ uniform random variables in $(0,1)$, each one associated with one block, and all blocks within one subinterval merge.
In the case of the symmetric coalescent, the paintbox construction can be reformulated as follows: when there are $b$ blocks, at rate $F(k)$, we distribute $b$ balls into $k$ boxes and blocks corresponding  to balls that are in the same box merge. For more details we refer the reader to the first chapter of \cite{Berestycki}.
This construction allows us to obtain a nice explicit formula for the transition rates.  
 
 \begin{proposition}
For each $b\ge 2$ and  $k_1, \ldots, k_r$  such that  $\sum_{i=1}^r k_i  = b$, we have
\begin{equation*}
\lambda_{b, (k_1, \dots, k_r)} \ = \ a\  \mathds{1}_{\{r = b-1, k_1 = 2\}} \ + \ \sum_{k\ge r} F(k) \frac{k!}{(k-r)!} \frac1{k^b},\end{equation*}
where $a = F(0)$.
\end{proposition}
This result is obtained from \eqref{ratesXi} as follows. The term $F(k)$ comes from $\Xi^0(d\zeta)/(\zeta,\zeta)$, while $k!/(k - r)!$ is the number of choices of $i_1, \dots, i_r$ and in this case $\zeta_{i_i}^{k_1}\dots\zeta_{i_r}^{k_r}$ equals $1/k^b$.
In other words, at each jump time, if the number of boxes is $k$ (chosen with respect to $F$), we choose $r$ \textit{ordered} boxes and we allocate the $b$ balls to these $r$ boxes ($k_1$ balls to the first box, $k_2$ to the second one, etc...).

Finally let us consider $\{N_t\}_{t \geq 0}$, the block-counting process of the symmetric coalescent and, for $i>j$, let us denote by $q_{ij}$ its transition rate from $i$ to $j$. 
Our next result is the symmetric coalescent version of Proposition 2.1 in \cite{GM}.
Let $W^{k,b}$ be the random variable corresponding to the number of non-empty boxes when allocating randomly $b$ balls into $k$ boxes, whose distribution can be found in  \cite{durrett}, proof of Theorem 3.6.10, page 172.

\begin{proposition}
We have 
\begin{equation*}
q_{ij} \ = \ a \binom{i}{2} \mathds{1}_{\{j = i-1\}} \ + \  \sum_{k\ge 1} F(k) \p(W^{k,i} = j) 
\end{equation*}
with 
\begin{equation*}
 \p(W^{k,i} = j)  \ = \ \binom{k}{j} \left(\frac{j}{k}\right)^i \sum_{r=0}^j (-1)^r \binom{j}{r} \left(1- \frac{r}{j}\right)^i. 
\end{equation*}
\label{ratesblockcounting}
\end{proposition}

The fact that the characterizing measure $S$ (or $F$) only puts weight on elements of $\Delta^{sym}$ simplifies a lot the global picture of the coalescence tree, even when starting from an infinite population. 
In particular, an $S$-coalescent is almost surely finite after the first coalescence (that is not a `Kingman type' coalescence).
However, the symmetric coalescent can {\it come down from infinity}. 
Denoting by $\{N_t\}_{t\geq0}$ the block-counting process of the symmetric coalescent and assuming that $N_0=\infty$, recall that a coalescent {comes down from infinity} if for every $t>0$, $N_t<\infty$ almost surely.
Observing that the time of the first (non `Kingman type') coalescence event is exponentially distributed, with parameter $\sum_{k=1}^\infty \frac{S(\xi^k)}{(\xi^k, \xi^k)} = \sum_{k=1}^\infty kS(\xi^k)$, it is straightforward to get the next result.
 \begin{proposition}\label{CDI}
 An $S$-coalescent comes down from infinity if and only if 
  $$
 S(\xi^0) >0 \ \textrm{ or } \ \sum_{k\ge1} k S(\xi^k)=\sum_{k\ge1} F(k)=\infty.
 $$
 \end{proposition}

 An interesting family of symmetric coalescents, that we will call $(\beta,S)$-coalescents, contains those characterized by $F(k) = k^{-\beta}$, for  $\beta >0$ (so that condition \eqref{Ffinita} is satisfied).
 By Proposition \ref{CDI}, a $(\beta,S)$-coalescent
comes down from infinity if and only if $0<\beta\le 1$. 
We now focus on the total coalescence rate when there are $n$ lineages, which is given by 
\begin{equation}\label{tcr}
\lambda_n=\sum_{r =1}^{n-1} \ \sum_{k_1, \dots, k_r, \sum k_i = n} \mathcal{N}(n, (k_1,\dots, k_r)) \lambda_{n, (k_1,\dots, k_r)},\end{equation}
where  $\mathcal{N}(n, (k_1,\dots, k_r))$ is the number of different simultaneous choices of a $k_1$-tuple, a $k_2$-tuple,... and a $k_r$-tuple from a set of $b$ elements. An explicit formula for this number can be found in \cite{S2000}, display $(3)$. 
\begin{proposition}\label{rate}
For the $(\beta, S)$-coalescent, with $\beta \in\  (0, 1)$, we have
$$\lim_{n\to\infty} n^{2(\beta-1)}\lambda_n= \frac{2^{\beta-1}\Gamma(\beta)}{1-\beta}.$$
For the $(1, S)$-coalescent, we have
$$\lim_{n\to\infty}\frac{\lambda_n}{\log n}=2.$$
\label{totalrate}
\end{proposition}

It is interesting to compare these asymptotics with other classical coalescents. For example, when $\beta \to 0$, the total coalescence rate becomes very close to the total coalescence rate of the Kingman coalescent, which is of order $n^2$. When $\beta\in(0,1/2]$ the total coalescence rate is very close to that of a $Beta(2-2\beta, 2\beta)$-coalescent, see Lemma 2.2 in \cite{DDSJ}. 
In particular, the rates of the $(1/2,S)$-coalescent have the same order than those of the Bolthausen-Sznitman coalescent.

\subsection{Genealogies of Wright-Fisher models with bottlenecks}
In this paper, we study four types of Wright-Fisher models with bottlenecks and their genealogies.
We establish the relations between forwards and backwards models via moment duality results.
In particular, when the bottlenecks are short, drastic and rare and $b_{i,N}N$ is distributed as $F^0$, a measure on $\N$, 
we prove (see Theorem \ref{SDEs3}) that the scaling limit, in the sense of weak convergence in the Skorokhod topology, of a subpopulation frequency is given by a Wright-Fisher diffusion with jumps 
$$dX_t \ = \ \sqrt{X_t(1-X_t)}dB_t \ + \ \int_{\N}  \int_{[0,1]^\N} \frac1k \sum_{i=1}^k \left( \mathds{1}_{\{u_i \le X_{t^-}\}} - X_{t^-}\right) \hat N(dt, dk, du),
$$where  $\{B_t\}_{t\geq0}$ is a standard Brownian motion and $\hat N$ is a compensated Poisson measure on $(0, \infty) \times \N \times [0,1]^{\N}$ with intensity $ds \otimes F^0(k) \otimes du$, where $du$ is the Lebesgue measure on $[0,1]^{\N}$.
The jump term  can be interpreted as follows. At rate $F^0(k)$ there is a bottleneck in which only $k$ individuals survive. The term `$\mathds{1}_{\{u_i \le X_{t^-}\}}$' is the probability that an individual chooses a type 1 parent (when choosing her parent uniformly from the generation before the bottleneck), and therefore $\frac1k \sum_{i=1}^k \mathds{1}_{\{u_i \le X_{t^-}\}}$ is the frequency of type 1 individuals after the bottleneck.
As we shall see in Section \ref{duality1}, this equation has a unique strong solution that is moment dual to the block-counting process of the symmetric  coalescent characterized by $F = \delta_0 + F^0$.
Duality relations between $\Xi$-coalescents and Wright-Fisher diffusions with jumps were established in \cite{BBMST}. 
Moment duality implies that the process counting the number of ancestors to a sample of individuals in the Wright-Fisher model with short drastic bottlenecks converges to the block-counting process of the symmetric coalescent, in the sense of finite dimensional distributions. 
In Section 3.1 of \cite{GS} the authors show that convergence in the $J_1$ Skorokhod topology of the forward frequency process to the solution of the above SDE implies convergence in $J_1$ of the process counting the number of ancestors to a sample.

A similar strategy is used in the case of  Wright-Fisher models with long bottlenecks.
However, the situation is different since the frequency process is not Markovian anymore, as the transition rates depend on whether the population is undergoing a bottleneck or not. 
Nevertheless, we still obtain a scaling limit, but in the sense of convergence in measure (as defined in Section \ref{topology}) over the Skorokhod space, to a diffusion with jumps.
Intuitively, only a measure zero set of points prevents  $J_1$ convergence. These are exactly the accumulation points of the times at which a bottleneck occurred in the discrete model. The method for proving this convergence relies on a new Skorokhod type metric that allows us to prove convergence in measure using standard arguments.
We prove that the diffusion with jumps  is moment dual to the block-counting process of a  $\Xi$-coalescent. In the case of long drastic bottlenecks, it is the \textit{drastic bottleneck coalescent} (see Definition \ref{longdrastic}) and in the case of long soft bottlenecks it is the \textit{subordinated Kingman coalescent} introduced in \cite{BBMST} and studied in Section \ref{sub-kingman}.
Again, moment duality implies convergence, in the sense of finite dimensional distributions, of the process counting the number of ancestors of a sample. 

In the case of Wright-Fisher models with soft drastic bottlenecks, a different strategy is used. Using M\"ohle's theorem \cite{mohle3}, we obtain that, under an appropriate time re-scaling, the (partition-valued) ancestral process converges to a time-changed Kingman coalescent. Again, moment duality implies the convergence in the sense of finite dimensional distributions of  the frequency process to a Wright-Fisher diffusion.
Table \ref{enbref} summarizes these results. 
\begin{table}
 \begin{center}
 \begin{tabular}{| c | c | c |}
  \hline
  & Drastic  & Soft \\
  \hline
  & &\\
Short & $S$-coalescent & Continuous time-rescaling of the Kingman coalescent \\
  & &\\
  \hline
    & &\\
Long & Drastic bottleneck coalescent & Subordinated Kingman coalescent \\
  & &\\
  \hline
\end{tabular}
\end{center}
\caption{Limiting genealogies for the different types of Wright-Fisher models with bottlenecks.}\label{enbref}
\end{table}

\subsection{Outline}
We start the core of this article by a complete study of the symmetric coalescent.
More precisely, in Section \ref{charac}, we prove Theorem \ref{propXiF}.
In Section \ref{ratelength}, we study asymptotics of the total coalescent rates (Proposition \ref{totalrate}) and the tree length in the special case of $(\beta,S)$-coalescents.
In Section \ref{duality1}, we establish a first duality result between the $S$-coalescent and the Wright-Fisher diffusion with short drastic bottlenecks.
In Section \ref{topology} we introduce the new Skorokhod type metric, that will be used in the last two sections, which are devoted to the study of other models with bottlenecks and their genealogies:
Section \ref{drastic} for long drastic bottlenecks and Section \ref{soft} for soft bottlenecks, where time-changed Kingman coalescents appear as limiting genealogies.

\section{The symmetric coalescent}
\label{Scoalescent}
We will start by considering bottlenecks that are drastic and short i.e. bottlenecks that only last for one generation and in which the population size during the bottleneck does not tend to infinity as $N\to \infty$. More precisely we consider the following model (that is a special case of Definition \ref{def2}).

\begin{definition}[Wright-Fisher model with short drastic bottlenecks]
Fix  $\alpha \in(0,1]$, $N \in \N$, $k^{(N)} \in (0, N^\alpha)$ and $F^0$ a probability measure on $\N$.  Let  $\{F_g\}_{g \in \Z_+}$ be a sequence of i.i.d. random variables of law $F^0$. 
Also, let $\{B^N_g\}_{g \in \Z_+}$ be a sequence of i.i.d. Bernoulli random variables of parameter $k^{(N)}/N^{\alpha}$. 
The Wright-Fisher model with short drastic bottlenecks is such that the sequence of population sizes $\{R^N_g\}_{g \in \Z_+}$ is given by $$\ R^N_g=N(1-B_g^N)+\min(N, F_g)B_g^N.$$
\label{shortdrastic1}
\end{definition}

\begin{remark}
In this case, the bottlenecks are short and  if the $i$-{th} bottleneck takes place during generation $g$, $b_{i,N}N = \min(N,F_g)$, which does not tend to infinity when $N$ goes to infinity, so the bottlenecks are drastic. 
In addition, $s_{i,N}$, the time between two bottlenecks follows a geometric distribution of parameter  $k^{(N)}/N^{\alpha}$, so {if $k^{(N)}=O(1)$, the expectation of $s_{i,N}/N$ is of order} $N^{\alpha}/N$. Thus, when $\alpha < 1$, the bottlenecks are frequent and when $\alpha = 1$ the bottlenecks are rare{; the main consequence  of this is that in the case of rare bottlenecks there is a Kingman/Wright-Fisher component, while in the case of frequent bottlenecks \textit{there is not enough time} for the Kingman part to be a part of the scaling limit.}
\end{remark}

As we will prove, when $N\to \infty$ and time is rescaled by $N^{\alpha}$, the genealogy of this model is described by the symmetric coalescent. It is now time to study this process. 
\subsection{Characterization}\label{charac}
Let us start with the proof of Theorem \ref{propXiF}.

\begin{proof}[Proof of Theorem \ref{propXiF}]
From \eqref{StoF}, we decompose $F$ (resp. $S$) into a `Kingman part' and a `simultaneous multiple collisions' part, i.e. $F = a \delta_0 + F^0$ where $a := F(0) \ge 0$ and $F^0(0) = 0$ (resp. $S = a \delta_{(0,0, \dots)} + S^0$).

We start by proving that any $\Xi$-coalescent that is characterized by a measure $S$ on $\Delta^{sym}$ as above is symmetric. 
We fix $b\ge 2$ and  $k_1, \ldots, k_r$ and $k'_1, \ldots, k'_r$ such that  $\sum_{i=1}^r k_i = \sum_{i=1}^r k'_i = b$. 
From Theorem 2 in \cite{S2000}, the transition rates can be written as follows: 
\begin{align*}
\lambda_{b, (k_1, \dots, k_r)} &= a\mathds{1}_{\{r = b-1\}}+ \int_{\Delta} \sum_{i_1 \ne \dots \ne i_r} \zeta^{k_1}_{i_1} \dots \zeta^{k_r}_{i_r} \  \frac{S^0(d\zeta)}{(\zeta, \zeta)}\\
&=a\mathds{1}_{\{r = b-1\}}+  \sum_{j = r}^\infty   \frac{j!}{(j-r)!} \left(\frac1j\right)^b F^0(j)\\
& = \lambda_{b, (k'_1, \dots, k'_r)}.
\end{align*}

Conversely, suppose that a $\Xi$-coalescent satisfies the symmetric condition on its transition rates.  We write $\Xi = a \delta_{(0,0, \dots)} + \Xi^0$.
For any $\zeta \in \Delta$, we set $\zeta_0 = 1-\sum_{i=1}^\infty \zeta_i $. We define
$$Z = \{\zeta \in \Delta, \ \exists j,i, \ \zeta_i > \zeta_j > 0 \}$$ 
and we assume that $\Xi^0(Z) >0$.
Using Theorem 2 in \cite{S2000}, we have
\begin{align*}
 \lambda_{4,(2,2)}  \ &= \ \int_\Delta \sum_{i_1\ne i_{2}}  \zeta_{i_1}^2  \zeta_{i_2}^{2} \frac{\Xi^0(d\zeta)}{(\zeta,\zeta)}
\  = \ 2 \int_\Delta \sum_{i_1< i_{2}} \ \zeta_{i_1}^2 \zeta_{i_2}^{2}\frac{\Xi^0(d\zeta)}{(\zeta,\zeta)}
\end{align*}
and 
\begin{align*}
 \lambda_{4,(3,1)} \ &= \  \int_\Delta \left( \sum_{i_1\ne i_{2}}  \zeta_{i_1}^3   \zeta_{i_2}  +  \zeta_0 \sum_{j} \zeta_j^3 \right)\frac{\Xi^0(d\zeta)}{(\zeta,\zeta)}\\
& \ge  \int_\Delta \left(\sum_{i_1 < i_{2}}  \zeta_{i_1}^3 \zeta_{i_2} + \sum_{i_1 < i_{2}} \zeta_{i_1}  \zeta_{i_2}^3 \right) \frac{\Xi^0(d\zeta)}{(\zeta,\zeta)}.
\end{align*}
So,
\begin{align*}
\lambda_{4,(3,1)} - \lambda_{4,(2,2)} \ &\ge \int_\Delta \left( \sum_{i_1 < i_{2}}  \zeta_{i_1}^3  \zeta_{i_2} + \sum_{i_1 < i_{2}}  \zeta_{i_1}  \zeta_{i_2}^3-2\sum_{i_1 < i_{2}} \zeta_{i_1} ^2 \zeta_{i_2}^2 \right) \frac{\Xi^0(d\zeta)}{(\zeta,\zeta)}\\
& = \int_\Delta \sum_{i_1 < i_{2}}  \zeta_{i_1}  \zeta_{i_2} (  \zeta_{i_1} - \zeta_{i_2})^2 \frac{\Xi^0(d\zeta)}{(\zeta,\zeta)} \\
& = \int_Z \sum_{i_1 < i_{2}}  \zeta_{i_1}  \zeta_{i_2} (  \zeta_{i_1} - \zeta_{i_2})^2 \frac{\Xi^0(d\zeta)}{(\zeta,\zeta)} >0,
\end{align*}
as the integrand is equal to zero on $\Delta \setminus Z$ and strictly positive on $Z$.
This cannot be true (as the coalescent is symmetric). So we need $\Xi^0(Z) =0$, i.e.  $\Xi^0$ can only take positive values on  $\Delta \setminus Z$, i.e. elements of $\Delta$ such that there exists $ 0< u \le 1$ with $\zeta = (u, u, \dots, u, 0, 0,\dots)$.

Now, we consider the set
 $$Z_0 = \{\zeta \in \Delta \setminus Z, \ \zeta_0 > 0 \} $$ 
and we  assume that  $\Xi^0(Z) = 0$ and $\Xi^0({Z_0}) >0$. We have 
$$ \lambda_{4,(2,2)} =  \int_{\Delta\setminus Z} \sum_{i_1\ne i_{2}}  \zeta_{i_1}^2  \zeta_{i_2}^{2} \frac{\Xi^0(d\zeta)}{(\zeta,\zeta)}$$
and $$\lambda_{4,(3,1)}  =  \int_{\Delta\setminus Z} \left( \sum_{i_1\ne i_{2}}  \zeta_{i_1}^3   \zeta_{i_2}  +  \zeta_0 \sum_{j} \zeta_j^3 \right)\frac{\Xi^0(d\zeta)}{(\zeta,\zeta)}.
$$
Recall that, if $\zeta \in \Delta \setminus Z$, then $\forall i \ge 1, \zeta_i = \zeta_1$ or $\zeta_i = 0$, so $\sum_{i_1\ne i_{2}}  \zeta_{i_1}^2  \zeta_{i_2}^{2} = \sum_{i_1\ne i_{2}}  \zeta_{i_1}^3   \zeta_{i_2} $.
So $ \lambda_{4,(2,2)} = \lambda_{4,(3,1)}$ if and only if
 $\Xi^0(Z_0) = 0$, which means that $\Xi^0$ can only put weight on elements of $(\Delta \setminus Z )\setminus Z_0$, i.e. elements of $\Delta^{sym}$. This completes the proof. 
\end{proof}

As we shall see in Section \ref{duality1}, the symmetric coalescent characterized by a measure $F = a \delta_0 + F^0$, where $F^0$ is a probability measure,  describes the genealogy of a Wright-Fisher model with short drastic bottlenecks parametrized by $\alpha \in (0,1]$,  $N$, $k^{(N)} = 1$ and $F^0$, in the limit when $N \to \infty$. The case $a = 0$ corresponds to frequent bottlenecks ($\alpha <1$) and the case $a = 1$  corresponds to rare bottlenecks ($\alpha =1$). In fact, when time is rescaled by $N^{\alpha}$, if the bottlenecks are frequent, in the limiting genealogy we only see coalescent events taking place during the bottlenecks, whereas if the bottlenecks are rare, there is a `Kingman part' in the limiting genealogy, corresponding to coalescence events taking place outside the bottlenecks. 
We will also discuss in this section a model where the limit genealogical process is a symmetric coalescent characterized by a measure $F$ that is not finite.
Indeed, we show that, when $N\to \infty$, the Wright-Fisher model with short drastic bottlenecks converges to a diffusion with jumps that is moment dual to the block-counting process of the symmetric coalescent.

 \subsection{Tree length and total coalescence rate of $(\beta, S)$-coalescents}\label{ratelength}

We now focus on the family of $(\beta,S)$-coalescents. In this case, the total coalescence rate \eqref{tcr} is
\begin{align}\label{lambdaen}
\lambda_n \ = \ \sum_{k = 1}^{\infty} k^{-\beta} \p (\mathcal{C}^k_n),
\end{align}
where $\mathcal{C}^k_n$ is the event that, in the paintbox construction with $k$ boxes and $n$ balls, there are at least two balls that are allocated to the same box. 
For $n>k$, $\p (\mathcal{C}^k_n) = 1$ and for $n\le k$, 
$$\p (\mathcal{C}^k_n) = 1 - \prod_{i = 2}^n \frac{k+1-i}{k}.$$
In fact, the probability of $\mathcal{C}^k_n$ is 1 minus the probability that $n$ successive balls are allocated to distinct boxes, which can be computed in the following way: the first ball is allocated to any box and then, for $2\le i\le n$,  the $i$-{th} ball is allocated to one of the $k+1-i$ empty boxes. 
We are now ready to prove Proposition \ref{totalrate}.

 \begin{proof}[Proof of Proposition \ref{totalrate}] 
We first treat the case $\beta\in(0,1)$.
Fix $0<\epsilon<1$. We divide \eqref{lambdaen} into two parts:
$$\lambda_n=\sum_{k = 1}^{\lfloor n^{1+\epsilon}\rfloor-1} k^{-\beta}\p (\mathcal{C}^k_n)
+\sum_{k =\lfloor n^{1+\epsilon}\rfloor}^\infty k^{-\beta}\p (\mathcal{C}^k_n).$$
For the second term, we have
 \begin{align*}
 \sum_{k = \lfloor n^{1+\epsilon}\rfloor}^{\infty} k^{-\beta}\left(1 - \prod_{i = 1}^{n-1} \left(1-\frac{i}{k}\right)\right)
 &\sim\sum_{k = \lfloor n^{1+\epsilon}\rfloor}^{\infty} k^{-\beta}\left(1 -\exp\left(-\sum_{i = 1}^{n-1} \frac{i}{k}\right)\right)\\
 &\sim\sum_{k = \lfloor n^{1+\epsilon}\rfloor}^{\infty} k^{-\beta}\left(1 -\exp\left(-\frac{n^2}{2k}\right)\right)\\
 &\sim n^{2-2\beta}\int_0^\infty x^{-\beta}\left(1 -\exp\left(-\frac{1}{2x}\right)\right)dx\\
 &= \frac{2^{\beta-1}\Gamma(\beta)}{1-\beta}n^{2(1-\beta)},
 \end{align*}
 where the last equality is obtained by integrating by parts and using the inverse-gamma distribution.
 
 For the first term, observe that  
 $$\sum_{k = 1}^{\lfloor n^{1+\epsilon}\rfloor} k^{-\beta}\p (\mathcal{C}^k_n)
 \le\sum_{k = 1}^{\lfloor n^{1+\epsilon}\rfloor} k^{-\beta}
 \sim n^{(1+\epsilon)(1-\beta)},$$
 which is negligible compared to the second term.

 Let us now suppose that $\beta=1$.
 We divide \eqref{lambdaen} into three parts (recall that $\p (\mathcal{C}^k_n)=1$ when $k\le n$).
 $$\lambda_n=\sum_{k = 1}^{n-1} k^{-1}
 +\sum_{k = n}^{\lfloor n^{1+\epsilon}\rfloor-1} k^{-1}\p (\mathcal{C}^k_n)
+\sum_{k = \lfloor n^{1+\epsilon}\rfloor}^\infty k^{-1}\p (\mathcal{C}^k_n).$$
The first term is obviously equivalent to $\log n$. 
The second term is clearly smaller than $\epsilon\log n$. Let us now find a lower bound
 \begin{align*}
 \sum_{k = n}^{\lfloor n^{1+\epsilon}\rfloor-1} k^{-1}\left(1 - \prod_{i = 1}^{n-1} \left(1-\frac{i}{k}\right)\right)
&\ge \sum_{k = n}^{\lfloor n^{1+\epsilon}\rfloor-1} k^{-1}\left(1 - \exp\left(\sum_{i = 1}^{n-1} \left(-\frac{i}{k}\right)\right)\right)\\
 &\sim\int_{n^{-1}}^{n^{\epsilon-1}} x^{-1}\left(1 -\exp\left(-\frac{1}{2x}\right)\right)dx\\
&=\int_{n^{1-\epsilon}}^{n} y^{-1}\left(1 -\exp\left(-\frac{y}{2}\right)\right)dy\\
 &= \log n\left(1 -\exp\left(-\frac{n}{2}\right)\right)-(1-\epsilon)\log n\left(1 -\exp\left(-\frac{n^{1-\epsilon}}{2}\right)\right)\\&+\frac{1}{2}\int_{n^{1-\epsilon}}^n \log y\exp\left(-\frac{y}{2}\right)dy\\
 &\sim \epsilon\log n.
\end{align*}

For the third term, we use similar computations as in the case $\beta<1$,
 \begin{align*}
 \sum_{k = \lfloor n^{1+\epsilon}\rfloor}^{\infty} k^{-1}\left(1 - \prod_{i = 1}^{n-1} \left(1-\frac{i}{k}\right)\right)
 &\sim\int_{n^{\epsilon-1}}^\infty x^{-1}\left(1 -\exp\left(-\frac{1}{2x}\right)\right)dx\\
&=\int_0^{n^{1-\epsilon}} y^{-1}\left(1 -\exp\left(-\frac{y}{2}\right)\right)dy\\
 &= (1-\epsilon)\log n\left(1 -\exp\left(-\frac{n^{1-\epsilon}}{2}\right)\right)+\frac{1}{2}\int_0^{n^{1-\epsilon}} \log y\exp\left(-\frac{y}{2}\right)dy\\
 &\sim (1-\epsilon)\log n.
 \end{align*}
This ends the proof.
\end{proof}

This result on the total coalescence rate allows us to give a first estimate of the tree length of the $(\beta, S)$-coalescent. Let $L_n$ be the sum of the lengths of all the branches of the tree obtained from a $(\beta,S)$-coalescent started with $n$ lineages and stopped at the first time when there is only one lineage.
\begin{corollary}
For $\beta \in (0,1)$, there exist two positive constants $C'_\beta$ and $c'_\beta$, that  only depend on $\beta$ (and not on $n$), such that for $n$ large enough,
$$c'_\beta n^{(2\beta-1) \vee 0} \ \le \ \E(L_n) \ \le \ C'_\beta n^{(2\beta) \wedge 1}.$$
For $\beta = 1$, there exists a constant $C'_1$ such that
$$\frac{n}{2\log n}(1+o(1))\ \le \ \E(L_n) \ \le \ C'_1 n.$$
\label{coroalpha}
\end{corollary}

As we can see in Figure \ref{lengthalpha}, this corollary provides better estimates of the tree length when $\beta$ is close to 1 or close to 0.
\begin{figure}
\begin{center}
\includegraphics[width=10cm]{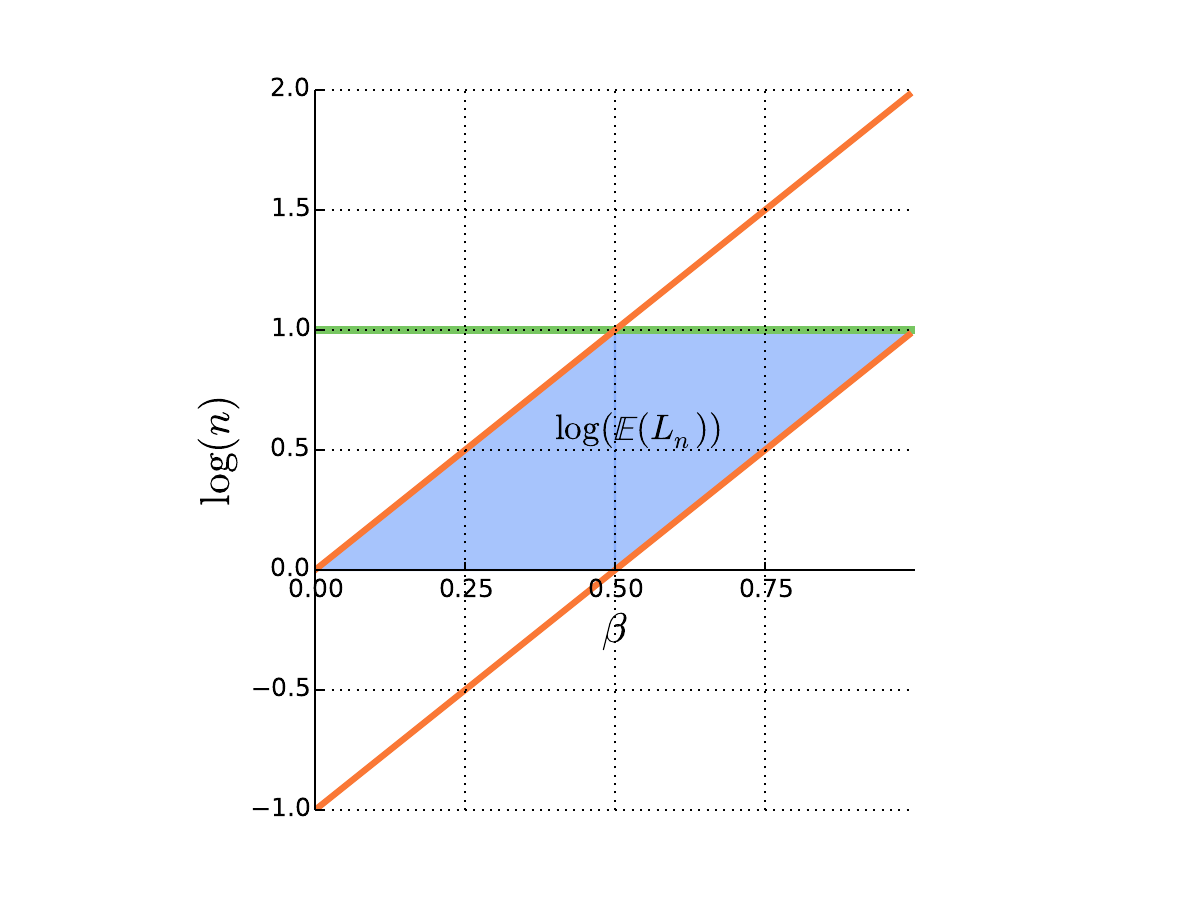}
\caption{Illustration of Corollary \ref{coroalpha}. The orange lines are $y= 2x-1$ and $y= 2x$, and the green line is $y=1$. The blue area is the region where $\log(\E(L_n))$ is located.}
\label{lengthalpha}
\end{center}
\end{figure}

  \begin{proof} 
We start by proving that, for any $\beta \in [0,1]$, the expected tree length is at most of order $n$. We consider the $S^1$-coalescent, characterized by $S^1(\xi^k) = \mathds{1}_{\{k = 1\}}$. First, one can easily show that the tree length of the $S^1$-coalescent is of order $n$ (it is a star-shaped coalescent). Second, the rate of events of size 1 in the $S^1$-coalescent and in the $(\beta, S)$-coalescent is the same and, for $k>1$, $S^1(\xi^k) = 0$ while, in the $(\beta, S)$-coalescent, $S(\xi^k) = k^\beta>0$. It is  not hard to construct a coupling between the two processes in such a way that the length of the $S^1$ coalescent is always larger than the length of the $(\beta, S)$-coalescent.

The expectation of the time to the first coalescence when there are $k$ lineages is $1/ \lambda_k$, so ${k}/{\lambda_k}$ is the expected length of a tree started with $k$ lineages and stopped at the first coalescence event. 
So, for $\beta \in (0,1)$, we have
 $$ \E(L_n) \ \le \ \sum_{k=2}^n  \frac{k}{\lambda_k}.$$
Recall that, in a coalescent where only two blocks can coalesce at a time (for example the Kingman coalescent), the sum on the right hand side would be the expected length, but in a coalescent with simultaneous multiple collisions we do not observe all the states $\{2, \dots, n\}$ for the block-counting process so it is only an upper bound. Using Proposition \ref{totalrate} for $\beta<1$, there exists a constant $c$ such that
\begin{align*}
 \E(L_n) \ &\le \frac{2^{\beta-1}\Gamma(\beta)}{1-\beta} \ \sum_{k=2}^n k^{2\beta-1}(1+o(1)) \\
& \le c \ \int_{1}^n t^{2\beta-1} dt  = \frac c{2\beta} (n^{2\beta} -1),
\end{align*}
which completes the proof of this first step.

\bigskip

 For $\beta \in (0,1]$, for the lower bound, we have
 $$ \E(L_n) \ \ge \  \frac{n}{\lambda_n},$$
 which is the length of the tree stopped at the first coalescence event.
When $0<\beta<1/2$ this lower bound is not interesting, as it is of order $n^{2\beta -1}$ and it decreases with $n$. But $\E(L_n)$ can always be bounded from below by a positive constant, which completes the proof.
\end{proof}
 
 \subsection{Duality with the Wright-Fisher model with short drastic bottlenecks}
\label{duality1}
We consider the Wright-Fisher model with short drastic bottlenecks from Definition \ref{shortdrastic1}.
Imagine that there are two types of individuals, $0$ and $1$, and each individual inherits the type of her parent.  We denote by $\{X^{N}_g\}_ {g\in\Z}$ the process corresponding to the frequency of type $1$ individuals in the population i.e., for any $g \in \Z_+$, 
$$X^N_g = \frac{\sum_{i=1}^{R_g^N} \mathds{1}_{\{(i,g)\textrm{ is of type 1}\}}}{R_g^N}. $$
As in the classical Wright-Fisher model, given {$R^N_{g+1}$ and} $X^N_g$, $R_{g+1}^NX^N_{g+1}$ follows a binomial distribution of parameters {$R_{g+1}^N$} and $X^N_g$. 
 In the following, `$\Longrightarrow$' denotes weak convergence in the Skorokhod on topology $D([0,1], \R_+)$.

\begin{theorem}
Let $F^0$ be a measure on $\N$ that fulfills condition \eqref{Ffinita}.
Fix  $\alpha \in(0,1]$ and $\gamma \in (0,\alpha/2)$.  
We consider the probability measure $ F_\gamma^N$ defined by 
$$ F_{\gamma}^N := \frac{\sum_{k=1}^{\lfloor N^\gamma \rfloor}F^0(k)\delta_k}{\sum_{k=1}^{\lfloor N^\gamma \rfloor}F^0(k)}.$$
Consider the sequence of processes  $\{X^N\}_{N \in \N}$, such that $X^N=\{X^{N}_g\}_ {g\in\Z_+}$ is the frequency process associated with the Wright-Fisher model with short drastic bottlenecks parametrized by $\alpha$,  $N$, $k^{(N)} = \sum_{k=1}^{\lfloor N^\gamma \rfloor}F^0(k)$ and $F_{\gamma}^N$ (from Definition \ref{shortdrastic1}).  Then,
$$
\{X^N_{\lfloor N^\alpha t \rfloor}\}_{t\ge0}\underset{N\to\infty}{\Longrightarrow} \{X_t\}_{t\ge0},
$$
where $\{X_t\}_{t\ge0}$ is the unique strong solution of the SDE
\begin{equation}
dX_t \ = \ \mathds{1}_{\{\alpha = 1\}} \sqrt{X_t(1-X_t)}dB_t \ + \ \int_{\N}  \int_{[0,1]^\N} \frac1k \sum_{i=1}^k \left( \mathds{1}_{\{u_i \le X_{t^-}\}} - X_{t^-}\right) \hat N(dt, dk, du),
\label{SDE1}
\end{equation}
where  $\{B_t\}_{t\geq0}$ is a standard Brownian motion and $\hat N$ is a compensated Poisson measure on $(0, \infty) \times \N \times [0,1]^{\N}$ with intensity $ds \otimes F^0(k) \otimes du$, where $du$ is the Lebesgue measure on $[0,1]^{\N}$.
\label{SDEs}
\end{theorem}

The same result holds if we consider $F^0$, a probability measure on $\N$ and the Wright-Fisher model with short drastic bottlenecks parametrized by $\alpha$, $N$, $k^{(N)} = 1$ and $F^0$. 

\begin{remark}
The definition of $F_{\gamma}^N$ ensures that $F_{\gamma}^N$ is a probability measure on $\N$ and $k^{(N)}/N^\alpha\in [0,1]$, so the Wright-Fisher model with short drastic bottlenecks is well-defined, at least for $N$ large enough. In fact, if $F^0$ satisfies  condition \eqref{Ffinita}, for $k$ large enough $F(k) <k$, so $k^{(N)} \le N^{2\gamma} + C$, where $C$ is a constant.
\label{rem3}
\end{remark}

In words,  the frequency process associated with the Wright-Fisher model with short drastic bottlenecks converges  to a diffusion with jumps that is similar to the frequency process associated with  $\Xi$-Fleming-Viot process (where the characterising measure $\Xi$ is the measure $S$ on $\Delta^{sym}$ that can be obtained from $F^0$ as in Theorem \ref{propXiF}). When the bottlenecks are frequent ($\alpha <1$), the limiting process is a pure jump process whereas,  when the bottlenecks are rare ($\alpha = 1$), we also have a diffusion term, which is a Wright-Fisher diffusion and corresponds to the evolution of the population outside the bottlenecks.
Before proving Theorem \ref{SDEs} we shall make sure that a solution to Equation \eqref{SDE1} exists.

\begin{lemma}
For any measure $F^0$  in $\N$ that satisfies condition \eqref{Ffinita}  and any $\alpha \in (0,1]$, there exists a unique strong solution to the SDE \eqref{SDE1}.
\label{lemmaexistence}
\end{lemma}

\begin{proof}
This result is a direct consequence of Lemma 3.6 in \cite{GS} (which is itself a consequence of Theorem 5.1 in \cite{Li}), applied to the measure $S$ on $\Delta^{sym}$ obtained from $F^0$ as in Proposition \ref{propXiF} and a  drift coefficient  equal to 0. 
\end{proof}

We are now ready to prove Theorem \ref{SDEs}. 

\begin{proof}[Proof of Theorem \ref{SDEs}]
The proof follows closely the proof of Proposition 3.4 in \cite{GS}. The idea is to prove the convergence of the generator of $\{X^N_{\lfloor N^\alpha t\rfloor}\}_{t\ge0}$ to the generator of $\{X_t\}_{t\ge0}$. Provided this claim is true, we can use Theorem 19.25 and 19.28 of \cite{Kallemberg} to prove the weak convergence in the Skorokhod topology.

From Lemma \ref{lemmaexistence}, $\{X_t\}_{t\ge0}$ exists and has generator
 $\mathcal{A}$. Its domain contains twice differentiable functions and  for a function $f \in C^2[0,1]$ and $x \in [0,1]$, we have
 \begin{align}
\mathcal{A}f(x) \ =& \ \mathds{1}_{\{\alpha = 1\}}\frac12x(1-x)f''(x) \ + \ \sum_{k\ge 1} F^0(k)  \E\left(f\left(\frac{\sum_{i=1}^{k}B_i^x}{k}\right) - f(x)\right),
\label{generatorA}
\end{align}
where the $B_i^x$'s are independent Bernoulli random variables of parameter $x$ and the second term is the generator of a $\Xi$-Fleming-Viot process, see for example formula (5.6) in \cite{BBMST} (applied to the measure $S$ associated with $F^0$).

For every $N \in \N$, let $\mathcal{U}^N$ be the transition operator associated with $X^N$ and define the operator 
\begin{equation}
\mathcal{A}^N := N^\alpha (\mathcal{U}^N - I),
\label{discretegenerator} 
\end{equation}
where $I$ is the identity operator (see Theorem 19.28 in \cite{Kallemberg}). $\mathcal{A}^N$ is referred to as the discrete generator of $\{X^N_{\lfloor N^{\alpha} t \rfloor}\}_{t\ge0}$. For any function $f\in C^2[0,1]$ in $x \in [0,1]$ we have
\begin{align}
\mathcal{A}^Nf(x) \ =& \ (1 - \frac{\sum_{k=1}^{\lfloor N^\gamma \rfloor}F^0(k)}{N^{\alpha}})  N^{\alpha}  \E \left( f\left(\frac{\sum_{i=1}^NB_i^x}{N}\right) - f(x)\right)  \label{gen1}\\ 
+& \  N^{\alpha} \frac{\sum_{k=1}^{\lfloor N^\gamma \rfloor}F^0(k)}{N^{\alpha}} \sum_{k= 1}^{\lfloor N^\gamma \rfloor} \frac{F^0(k)}{\sum_{k=1}^{\lfloor N^\gamma \rfloor}F^0(k)}  \E\left(f\left(\frac{\sum_{i=1}^{\min(N,k)}B_i^x}{\min(N,k)}\right) - f(x)\right). \label{gen2}
\end{align}
First, we study part \eqref{gen1}. Following Remark \ref{rem3}, the prefactor converges to 1.  When $\alpha = 1$, it is well known that \eqref{gen1} converges uniformly as $N\to \infty$ to $\frac12x(1-x)f''(x)$, which is the generator of the Wright-Fisher diffusion (see for example Chapter 2 and Theorem 3.6 in \cite{E}). 
When $\alpha <1$ this term becomes of order $N^{\alpha-1}$ and therefore converges to $0$. 
Second, it is easy to see that part \eqref{gen2} converges when $N\to \infty$ to the second term of $\mathcal{A}$ in  \eqref{generatorA}.
Combining these two results, we have $\mathcal{A}^Nf \to \mathcal{A}f$ uniformly. 
\end{proof}

Let us consider $\{N_t\}_{t\ge0}$, the block-counting process of the symmetric coalescent characterized by $F = \mathds{1}_{\{\alpha = 1\}} \delta_0 + F^0$ (or by the associated measure $S$ on $\Delta^{sym}$, as defined in Proposition \ref{propXiF}).
 We have the following duality relation between the block-counting process of the symmetric coalescent and  $\{X_t\}_{t\ge0}$, the unique strong solution of \eqref{SDE1}.
\begin{theorem}
For every $x \in [0,1], \ n \in \N$, we have
$$\E(X_t^n \vert X_0 = x) \ = \ \E (x^{N_t} \vert N_0 = n).$$
\label{momentduality}
\end{theorem}
This is a special case of Proposition 3.8  in \cite{GS}, but we find the proof of this particular case instructive and for the sake of completeness we include it here.
\begin{proof}
Recall that, from Proposition \ref{ratesblockcounting}, for any function $h: \N \to \R$, $\mathcal{G}$ the infinitesimal generator of the block-counting process $\{N_t\}_{t\ge0}$ is given by
\begin{equation}
\mathcal{G}h(n) := a \binom{n}{2} \left( h(n-1) - h(n) \right)  \ + \ \sum_{k\ge1} F^0(k) \sum_{j=1}^{n-1} \p(W^{k,n} =j) \left(h(j) - h(n)\right).
\label{generatorG}
\end{equation}
We first consider the case $\alpha <1$.
We use Lemma 4.1 in \cite{GS}, which states that generator $\mathcal{A}$ of  $\{X_t\}_{t\ge0}$, applied to a function $f \in C^2[0,1]$ admits the following representation
\begin{align*}
\mathcal{A}f(x) \ &= \ \frac{S(\Delta_{S})}{2} \E\left[ \frac{(-x + \sum_{i=1}^{\infty} Z_i B_i^x)(\sum_{i=1}^{\infty} Z_i B_i^x)}{(\sum_{i=1}^{\infty} Z_i^2 )} f''(x(1-V) + UV\sum_{i=1}^{\infty} Z_i B_i^x) \right]
\end{align*}
where $Z = (Z_1, Z_2, \dots)$ is $S$-distributed, $\{B_i^x\}_{i \in \N}$ is a sequence of i.i.d. Bernoulli random variables with parameter $x$, $U$ is uniform in $[0,1]$, $V$ is $Beta(2,1)$ in $[0,1]$ and $Z, \{B_i^x\}, U$ and $V$ are independent.
Using the definition of $S$, this can be rewritten as
\begin{align*}
\mathcal{A}f(x) \ &= \ \frac{\sum_{k\ge 1} F^0(k)/k}{2} \ \E\left[ \frac{(-x + \frac1K \sum_{i=1}^{K}  B_i^x)(\frac1K\sum_{i=1}^{K} B_i^x)}{(1/K)} f''(x(1-V) + UV\frac1K\sum_{i=1}^{K}B_i^x) \right] 
\end{align*}
where the expectation is taken with respect to $K$, a random variable such that for $ k \ge 1$, $$\ \p(K=k) = \frac{F^0(k)/k}{\sum_{k\ge 1} F^0(k)/k}.$$
Integrating by parts,
\begin{align*}
\mathcal{A}f(x) \ &= \ \frac{\sum_{k\ge 1} F^0(k)/k}{2} \ \E\left[ \frac{(-x + \frac1K \sum_{i=1}^{K}  B_i^x)}{V(1/K)} \left( f'(x(1-V) + V\frac1K\sum_{i=1}^{K}B_i^x) -f'(x(1-V))\right)\right].
\end{align*}
Conditioning on the value of $V$, we have
\begin{align*}
\E\left[ \frac{(-x + \frac1K \sum_{i=1}^{K}  B_i^x)}{V(1/K)} f'(x(1-V))\ \right] = 0.
\end{align*}
Again, following closely the proof of Lemma 4.1 in \cite{GS}, we calculate the expectation with respect to $V$, 
\begin{align*}
\mathcal{A}f(x) \ &= \ \frac{\sum_{k\ge 1} F^0(k)/k}{2} \ \E\left[ \frac{(-x + \frac1K \sum_{i=1}^{K}  B_i^x)}{V(1/K)} f'(x(1-V) + V\frac1K\sum_{i=1}^{K}B_i^x)\right] \\
   &= \ \frac{\sum_{k\ge 1} F^0(k)/k}{2} \ \E\left[ \int_0^1 K\frac{(-x + \frac1K \sum_{i=1}^{K}  B_i^x)}{s}  f'(x(1-s) + \frac{s}{K}\sum_{i=1}^{K}B_i^x) 2sds \right ]  \\  
&=  \ \left({\sum_{k\ge 1} F^0(k)/k}\right) \    \E\left[ K \int_0^1 (-x + \frac1K \sum_{i=1}^{K}  B_i^x)  f'(x(1-s) + \frac{s}{K}\sum_{i=1}^{K}B_i^x) ds \right ]  \\  
&=  \ \left({\sum_{k\ge 1} F^0(k)/k}\right) \ \E\left[K\left( f(\frac1K\sum_{i=1}^{K}B_i^x) - f(x)\right)\right],
\end{align*}
where the last equality comes by integration by parts. 
Now, we consider the function $h$ on $[0,1]\times \N$ such that $h(x,n) = x^n$. We fix $n \in \N$ and we apply the generator $\mathcal{A}$ to $h$, seen as a function of $x$,
\begin{align}
\mathcal{A}h(x) \ &=  \ \left({\sum_{k\ge 1} F^0(k)/k}\right) \ \E\left[K\left( (\frac1K\sum_{i=1}^{K}B_i^x)^n - x^n\right)\right] \nonumber \\
&=  \ \left({\sum_{k\ge 1} F^0(k)/k}\right) \ \E\left[K\left( \frac1{K^n} \sum_{k_1+ \dots + k_K =n} \binom{n}{k_1, \dots, k_K} \prod_{i=1}^K(B_i^x)^{k_i} - x^n\right)\right]\nonumber\\
&=  \ \left({\sum_{k\ge 1} F^0(k)/k}\right) \ \E\left[ K\left( \sum_{k_1+ \dots + k_K =n} \binom{n}{k_1, \dots, k_K}\frac1{K^n}  \left(x^{\sum_{i = 1}^K \mathds{1}_{\{k_i>0\}}} - x^n\right)\right)\right], 
\label{dualitybinom1}
\end{align}
where in the last line we use the fact that  for $k_1, \dots, k_K \in \N$,  $$\E\left( \prod_{i=1}^K (B_i^x)^{k_i}\right) = x^{\sum_{i = 1}^K \mathds{1}_{\{k_i>0}\}}. $$
Consider the random variable $W^{k,n}$ defined in Proposition \ref{ratesblockcounting}. For any $\kappa \in\N$, we have
\begin{align}
\sum_{k_1+ \dots + k_\kappa =n} \binom{n}{k_1, \dots, k_\kappa}\frac1{\kappa^n}  \left(x^{\sum_{i = 1}^\kappa \mathds{1}_{\{k_i>0}\}} - x^n\right) &= \E \left[  x^{W^{\kappa,n}} - x^n\right] \ =: \ E_\kappa.
\label{dualitybinom2}
\end{align}
So we have, 
\begin{align*}
\mathcal{A}h(x) \ &=  \ \left({\sum_{k\ge 1} F^0(k)/k}\right) \ \E\left[ KE_K\right] \\
&=  \ \left({\sum_{k\ge 1} F^0(k)/k}\right) \ \sum_{k\ge1} \frac{F^0(k)/k}{\sum_{k\ge 1} F^0(k)/k} kE_k \\
&=   \sum_{k\ge1} F^0(k) E_k \\
&=   \sum_{k\ge1} F^0(k) \sum_{j = 1}^n \p(W^{k,n} =j) \left(x^{j} - x^n\right)  = \mathcal G h(n),
\end{align*}
where in the last line $h$ is seen as a function of $n$ and $\mathcal G$ is the generator of the block-counting process of the symmetric coalescent defined in \eqref{generatorG}, in the case $a = 0$. 
This completes the proof for the case $\alpha <1$.

If $\alpha = 1$ and $f$ is defined as previously, we have
\begin {align*}
\mathcal{A}f(x) \ = \ \mathcal{A}_1f(x) +\mathcal{A}_2f(x)
\end{align*}
where $\mathcal{A}_2$ is the infinitesimal generator of $\{X_t\}_{t\ge0}$ for the case $\alpha <1$ and $\mathcal{A}_1$ is the generator of the Wright-Fisher diffusion. Let $\{K_t\}_{t\ge0}$ be the block-counting process of the Kingman coalescent and $\{Y_t\}_{t\ge0}$ be the classical Wright-Fisher diffusion. We recall the well known moment duality between these two processes,
\begin{equation}\E(Y_t^n \vert Y_0 = x) \ = \ \E (x^{K_t} \vert K_0 = n).\label{classicalduality}\end{equation}
This implies that
\begin {align*}
\mathcal{A}_1h(x,n) \ &= \ \frac12x(1-x)\frac{\partial^2 h(x,n)}{\partial x} \\
&= \ \binom{n}{2} \left( h(x,n-1) - h(x,n) \right),
\end{align*}
where the last line corresponds to the generator of  $\{K_t\}_{t\ge0}$.
Combining this with the result for the case $\alpha <1$, we have
\begin {align*}
\mathcal{A}h(x,n) \ &= \ \binom{n}{2} \left( x^{n-1} - x^n \right) + \sum_{k\ge1} F^0(k) \sum_{j\le k} \p(W^{k,i} =j) \left(x^{j} - x^n\right) \\
& = \mathcal{G} h(x,n),
\end{align*}
where $\mathcal{G}$ is the infinitesimal generator of the block-counting process of the symmetric coalescent for the case $a = 1$. This completes the proof.
\end{proof}

\section{A topological interlude}
\label{topology}
As explained in the introduction, in the case of the Wright-Fisher model with long drastic bottlenecks and long soft bottlenecks, we also want to prove the convergence of the frequency process to a diffusion with jumps, which is moment dual to the block-counting process of a bottleneck coalescent.
In these cases, the frequency process is not Markovian (the transition rates depend on whether the population is undergoing a bottleneck or not)  and can have important fluctuations during the bottlenecks (see Figure \ref{nonmarkovian} for an illustration), that prevent the convergence in the Skorokhod $J_1$ (and $M_1$) topology. 
However, the points that prevent this  convergence are exactly the accumulation points of the times at which a bottleneck occurred in the discrete models, that have Lebesgue measure 0 when time is rescaled by $N^\alpha$ and $N$ tends to infinity.  

For $T>0$, we denote by $D[0,T]$ the space of real-valued c\`adl\`ag functions defined on $[0,T]$. We will introduce a new metric for convergence in measure on  $D[0,T]$ (see Meyer and Zheng \cite{MZ}). A sequence of c\`adl\`ag functions converges in measure to another c\`adl\`ag function if for any $\epsilon>0$, the Lebesgue measure of the set of points for which the distance to the limit is higher than $\epsilon$ converges to 0. 
We say that a sequence of stochastic processes converges in measure to another process on $D[0,T]$ if it converges weakly, as probability measures on $D[0,T]$, in the topology induced by convergence in measure.
Lemma 1 in \cite{MZ} states that convergence is measure is equivalent to convergence in the pseudopath space i.e. that $x_n \to x$ in measure on $D[0,T]$ if and only if for any continuous bounded function $g:[0,T]\times \R \mapsto \R$, 
 \begin{eqnarray*}
 \lim_{n\rightarrow\infty}\int_0^T g(s,x_n(s))ds&=&\int_0^T g(s,x(s))ds, 
\end{eqnarray*}
and $x_n(T) \to x(T)$. This is also known as convergence in the Meyer-Zheng topology.
 
 Our metric  has the advantage of being a slight modification of the Skorokhod $J_1$ distance \cite{S56} and thus some techniques to manipulate it are well known.
Our method  can be generalized to the study of other processes that can have strong fluctuations (sparks) only in a set of timepoints that has Lebesgue measure 0 in the limit. 
Let us denote by $||\cdot||$ the uniform norm on $D[0,T]$.

 \begin{definition}
 Fix $T>0$. Let $x_1,x_2\in D[0,T]$ and $\mathcal{I}$ the set of finite unions of c\`adl\`ag intervals on $[0,T]$ i.e. 
 $$
 \mathcal{I}=\{I=\cup_{i=1}^n [a_i,b_i): n\in\N, 0\leq a_i< b_i\leq T\}.
 $$
We say that a function is a Skorokhod reparameterisation of time (SRT) if  it is strictly increasing, continuous with a continuous inverse and we define
 $$
 \mathcal{F}=\{f:[0,T]\mapsto [0,T]: f \text{ is a SRT}\}.
 $$
 Let $\lambda$ be the Lebesgue measure on $[0,T]$ and $Id$ the identity map in $[0,T]$. For any $x_1,x_2\in D[0,T]$, define 
 $$
 d_\lambda(x_1,x_2)=\inf_{A\in\mathcal{I}, f\in\mathcal{F}}\{||\mathds{1}_A(x_1-x_2\circ f)||\vee||Id-f||\vee\lambda([0,T]/A)\vee |x_1(T)-x_2(T)|\}.
 $$
 \end{definition}
The topology induced by $d_\lambda$ is separable. Indeed, the set of $\mathbb{Q}$-valued staircase functions with discontinuities in $\mathbb{Q}$ is a countable and dense set.
However the resulting metric space is not complete. As a counterexample, one can think of the sequence $f_n = \sum_{i=1}^n(-1)^i \mathds{1}_{[1/(i+1), 1/i)}$ and its behavior at 0.

 \begin{proposition}
 The mapping $d_\lambda$ is a metric in the  Skorokhod space $D[0,T]$ for every $T>0$ and if $\lim_{n\rightarrow 0} d_\lambda(x_n,x)=0$, then $x_n\rightarrow x$ in measure on the Skorokhod space.
 \end{proposition}
 \begin{proof}
 It is clear that, if $x_1=x_2$, then $d_\lambda(x_1,x_2)=0$. 
 Now let us prove the converse. 
  Assume that $x_1\neq x_2$. If $|x_1(T)-x_2(T)|>0$ then $d_\lambda(x_1,x_2)>0$. 
 Now, if $|x_1(T)-x_2(T)|=0$, then there exists a point $\tau\in [0,T)$ such that $|x_1(\tau)-x_2(\tau)|>0$. 
 Let us call $y=|x_1-x_2|$ and observe that $y\in D[0,T] $. Using the right continuity of  $y$ at the point $\tau$, we know that there exist $\epsilon,\delta>0$ such that $\forall t\in[\tau,\tau+\delta), \  y(t)>\epsilon$. Let $A\in\mathcal{I}$. If $[\tau,\tau+\delta)\subset A^c$ then $\lambda(A^c)>\delta$. Otherwise, $||\mathds{1}_A(x_1-x_2)||>\epsilon$. Now, take $f\in\mathcal{F}$ such that $||Id-f||<\delta/2$, then there exists $\bar\tau\in [\tau,\tau+\delta) $ such that for all $t\in[\bar\tau,\bar\tau+\delta/2)$, $\bar y(t)=|x_1(t)-x_2(f(t))|>\epsilon.$ Repeating the argument,  we conclude that $d_\lambda(x_1,x_2)>\min\{\delta/2,\epsilon\}$, and thus $x_1=x_2$ if and only if $d_\lambda(x_1,x_2)=0$. 
 
Symmetry follows by the same arguments used by Skorokhod in the case of $J_1$ \cite{S56}. In fact $f \in \mathcal{F}$ implies that $f(0)=0$ and $f(T)=T$ and that $f^{-1}\in \mathcal{F}$.Using this observation, it is easy to see that $d_\lambda(x_1,x_2)=d_\lambda(x_2,x_1)$.

To show that the triangle inequality holds, let $x_1, x_2, x_3\in D[0,T]$ and observe that for any $A,B\in\mathcal{I}$ and $f_2,f_3\in\mathcal{F}$ it holds that $A\cup B\in\mathcal{I}$ and $f_1:=f_3\circ f_2\in \mathcal{F}.$ The triangle inequality follows from four observations (most of them due to Skorokhod \cite{S56}).

First,  
 \begin{eqnarray*}
||\mathds{1}_{A\cup B}(x_1-x_2\circ f_1)||&\leq &||\mathds{1}_{A\cup B}(x_1-x_3\circ f_2)||+||\mathds{1}_{A\cup B}(x_3\circ f_2-x_2\circ f_3\circ f_2)||\\
&\leq& ||\mathds{1}_{A}(x_1-x_3\circ f_2)||+||\mathds{1}_{ B}(x_3-x_2\circ f_3)||.
 \end{eqnarray*}
Second,  \begin{eqnarray*}
||Id-f_1||&\leq& ||Id-f_2||+||Id-f_3||.
 \end{eqnarray*}
  Third, 
  \begin{eqnarray*}
\lambda((A\cup B)^c)&\leq&\lambda(A^c)+\lambda( B^c).
 \end{eqnarray*}
 And finally, 
  \begin{eqnarray*}
|x_1(T)-x_2(T)|\leq |x_1(T)-x_3(T)|+|x_3(T)-x_2(T)|.
 \end{eqnarray*}
 Putting all these observations together, we conclude that the triangle inequality holds, and thus $d_\lambda$ is a metric.
 
 Finally, we need to prove that convergence according to $d_\lambda$ implies convergence in the Meyer-Zheng topology. 
Let $x,x_1,x_2,...\in D[0,T] $ and let $g:[0,T]\times \R \mapsto \R$ be a continuous bounded function. Assume that $\lim_{n\rightarrow \infty}d_{\lambda}(x_n,x)=0.$ For any $n \in \N$, there exists $\epsilon_n \to 0$, $A_n\in \mathcal{I}$ and $f_n\in\mathcal{F}$ such that
 $$
||\mathds{1}_{A_n}(x_n-x\circ f_n)||\vee||Id-f_n||\vee\lambda(A_n^c)\vee |x(T)-x_n(T)|<d_{\lambda}(x_n,x)+\epsilon_n,
 $$
 which implies that
  $$
\lim_{n\to\infty}||\mathds{1}_{A_n}(x_n-x\circ f_n)||\vee||Id-f_n||\vee\lambda(A_n^c)\vee |x(T)-x_n(T)|=0.
 $$
Since $g$ is continuous, bounded and $||Id - f_n|| \to 0$, we have 
\begin{eqnarray*}
 \lim_{n\to\infty} \int_0^T g(s,x_n(s))ds&=& \lim_{n\rightarrow\infty}\int_{A_n} g(s,x_n(s))ds+  \lim_{n\to\infty} \int_{A_n^c} g(s,x_n(s))ds.
\end{eqnarray*}
First, as  $d_\lambda(x_n, x) \to 0$ and $\epsilon_n \to 0$ , we have $\lambda(A_n^c) \to 0$ and $g$ is bounded, so
\begin{eqnarray*}
 \lim_{n\to\infty}  \int_{A_n^c} g(s,x_n(s))ds  = 0.
\end{eqnarray*}
Second, 
\begin{eqnarray*}
\int_{A_n} g(s,x_n(s))ds = \int_{A_n} g(f_n(s),x(f_n(s)))ds + \int_{A_n} \left (g(s,x_n(s))- g(f_n(s),x(f_n(s)))\right)ds,
\end{eqnarray*}
where 
\begin{eqnarray*}
|  \int_{A_n} \left (g(s,x_n(s))- g(f_n(s),x(f_n(s)))\right)ds| \le \sup_{s \in A_n}\{|g(s,x_n(s))- g(f_n(s),x(f_n(s))|\} T, 
\end{eqnarray*}
and as $g$ is bounded, continuous, $A_n$ is relatively compact and $||Id - f_n|| \to 0$, 
\begin{eqnarray*}
\lim_{n\to \infty} \sup_{s \in A_n}\{|g(s,x_n(s))- g(f_n(s),x(f_n(s))|\} = 0.
\end{eqnarray*}
Finally, using the fact that $f_n^{-1} \in \mathcal{F}$, and $\lim_{n\to\infty}||Id-f_n|| =0$, which implies that  $\lim_{n\to\infty}||Id-f^{-1}_n|| =0$, 
\begin{eqnarray*}
\lim_{n\to \infty} \int_{A_n} g(f_n(s),x(f_n(s)))ds  &=& \lim_{n\to \infty} \int_{f_n^{-1}(A_n)} g(t, x(t)) (f_n^{-1})'(t) dt \\
&=& \int_0^T g(t, x(t))dt.
\end{eqnarray*}
Combining these facts, we conclude that $x_n\rightarrow x$ in the pseudopath space, which completes the proof. \end{proof}

 \begin{remark}
 Convergence in $J_1$ implies convergence according to $d_\lambda$. To see this, take $A=[0,T]$.
  
 Convergence in the sense of $d_\lambda$ is close to convergence in measure. However, it is easy to see that it is not equivalent. The reason is that the difference of the functions in the point $\{T\}$ is crucial in order to have convergence according to $d_\lambda$ and $\{T\}$ is clearly a set of Lebesgue measure zero. It seems feasible to modify $d_\lambda$ in order to have the equivalence, but we decide to stay with $d_\lambda$ as it is because it is a minimalistic modification of the Skorokhod $J_1$ distance. Simply removing the term $|x_1(T)-x_2(T)|$ would cause that the paths that differ only on the last point would be at distance zero and then $d_\lambda$ would only be a pseudometric.
 \end{remark}

\section{Coalescents with drastic bottlenecks}
\label{drastic}
\subsection{The drastic bottleneck coalescent}
\label{sect212}
Now we consider bottlenecks that are drastic but can last for several generations. As we shall see, when the bottlenecks last for more than one generation the genealogy is not described by a symmetric coalescent anymore. 

\begin{definition}[Wright-Fisher model with long drastic bottlenecks]
Fix  $\alpha \in(0,1]$, $\eta >0$, $N \in \N$ and $F^0$ and $\mathrm{L}$ two  probability measures in $\N$.   
Let  $\{F_i\}_{i \in \N}$ be a sequence of i.i.d. random variables of law $ F^0$. 
Let $\{l_{i,N}\}_{i\in \N}$ and $\{s_{i,N}\}_{i\in \N}$ be two sequences of independent positive random variables such that  for all $ i \ge 1,  \ l_{i,N} $ converges in distribution to  $\mathrm{L}$ 
 and $s_{i,N}$ follows a geometric distribution of parameter $\eta/N^\alpha$.
In the Wright-Fisher model with long drastic bottlenecks  the sequence of population sizes $\{R^N_g\}_{g \in \Z_+}$ is given by
 $$
R^N_g \ = \ \left\{ 
\begin{array}{ll}
\min(F_m, N)  \textrm{ if } \sum_{i=1}^{m-1}(s_{i, N} + l_{i,N}) + s_{m,N} < g  \le  \sum_{i=1}^{m}(s_{i, N} + l_{i,N}) \\
N \textrm{ otherwise}
\end{array} 
\right.
$$
\label{longdrastic}
\end{definition}

\begin{remark}
As $\mathrm{L}$ does not depend on $N$, $\mathrm{L}/N^\alpha \to 0$ in distribution, which ensures that even if the bottlenecks last for several generations, their duration is negligible when time is rescaled by $N^\alpha$. Also, in that time scale, when $ N \to \infty$,  the distribution of the time between two bottlenecks converges to an exponential distribution of parameter $\eta$.
\end{remark}

In the limit when $N\to \infty$, the genealogy of this model can be described by the drastic bottleneck coalescent that we now define. As for the symmetric coalescent, the idea is that, in the genealogy there is a `Kingman part' corresponding  to what happens outside the bottlenecks (where the population size goes to infinity) and a `simultaneous multiple collisions' part corresponding to what happens during the bottlenecks.

To define this type of event, 
we start by fixing $k,g\in\N$ and considering the (partition-valued) ancestral process of a classical Wright-Fisher model with constant population size $k$, running for $g$ generations. 
The blocks obtained are given labels in $[k]$.
The block labelled $i$ contains all the descendants of individual $(i,0)$.
We then define the following random variables:
\begin{itemize}
\item $K^{k,g,b}$ is the number of ancestors of a sample of size $b\le k$.
\item $(A^{k,g}_1,  \dots, A^{k,g}_k)$ are the family sizes: $A^{k,g}_i$ is the size of the block labelled $i$ and $A^{k,g}_i = 0$ if there is no block labelled $i$. In other words, $A^{k,g}_i$ is the number of descendants of individual $i$ after $g$ generations. We denote by $\mathrm{A}^{k,g}$ the distribution, in $E^k=\{(i_1,\dots.i_k),\sum i_j=k\}$, of $(A^{k,g}_1,  \dots, A^{k,g}_k)$.
\item Let $V^{k,n}_i$ denote the number of balls allocated to box $i$ in the paintbox construction of the symmetric coalescent. We define a biased version of $\mathrm{A}^{k,g}$ as follows: For  $n \in \N \cup \infty$ and $i\in \{1, \dots, k\}$, $$\bar A^{k, g, n}_i \ = \ \sum_{j = 1}^k V^{k,n}_j \mathds{1}_{\{j \textrm{ belongs to the block labelled } i\}}. $$
\end{itemize}

\begin{definition}[The drastic bottleneck coalescent]
Fix $F^0$ and $\mathrm{ L}$ two probability measures in $\N$ and $\eta>0$.
The drastic bottleneck coalescent is defined by the following transition rates. 
For each $b\ge 2$ and  $k_1, \ldots, k_r$  such that  $\sum_{i=1}^r k_i  = b$,  each $[b,(k_1,\dots, k_r)]$-collision happens at rate
\begin{align*}
\lambda_{b, (k_1, \dots, k_r)} \ =&\ a \ \mathds{1}_{\{r = b-1, k_1 = 2\}} \nonumber \\
 + \ &  \mathcal{N}(n, (k_1,\dots, k_r))^{-1} \eta \sum \limits_{k\ge r} F^0(k) \sum \limits_{g\ge1} \mathrm{L}(g)  \sum \limits_{i_1\ne \dots \ne i_r} \p(\bar A^{k, g-1, b}_{i_1} = k_1, \dots,\bar A^{k, g-1, b}_{i_r} = k_r).
\end{align*}
\end{definition}
In words, there are $b$ lineages and a bottleneck of size $k$ and duration $g$ occurs. First each one of the $b$ lineages chooses a parent amongst the $k$ individuals of the last generation of the bottleneck. There remain $W^{k,b}$ lineages (each one containing $V^{k,b}_1, \dots, V^{k,b}_k$ original lineages). Then the bottleneck still lasts for $g-1$ more generations, in which the lineages merge as in a Wright-Fisher model with population size $k$ (see Figure \ref{figlineagesbot} for an illustration). 
\begin{figure}
\begin{center}
\includegraphics[width=15cm]{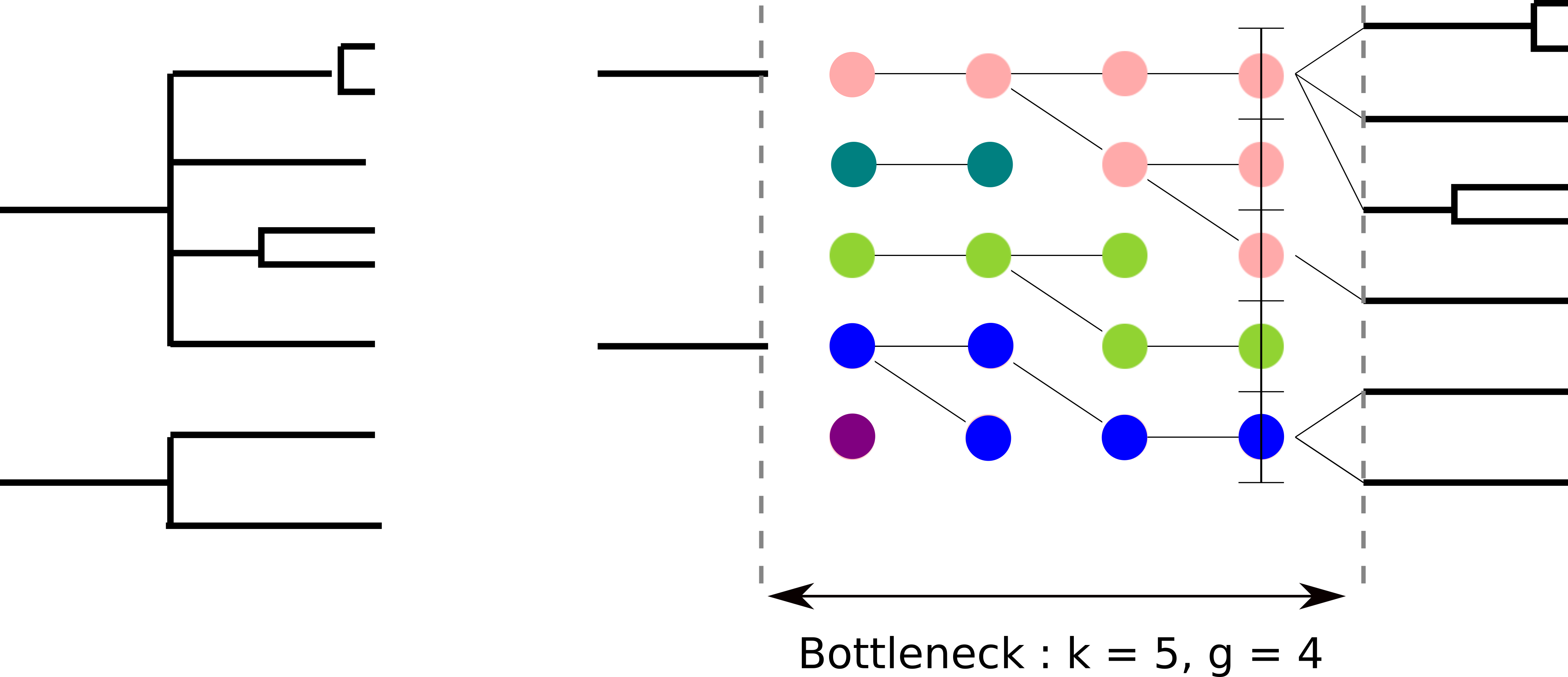}
\caption{A $[6, (4,2)]$-collision ilustrated. Before the bottleneck, only `Kingman type' mergers occur. Each of the 6 remaining lineages is allocated to one of the 5 individuals of the last generation of the bottleneck, and we have $V^ {5,6} = (3,0,1,0,2)$. Then, the system evolves for 3 more generations as a Wright-Fisher model. The family sizes are $A^{5,4}_1 = 3, \ A^{5,4}_2 = 0, \ A^{5,4}_3 = 1, \ A^{5,4}_4 = 1$ and $A^{5,4}_5 = 0$. The biased family sizes are $\bar A^{5,4,6}_1 = 4, \ \bar A^{5,4,6}_4 = 2$ and $\bar A^{5,4,6}_i = 0$ otherwise. }
\label{figlineagesbot}
\end{center}
\end{figure}
As a consequence, we have the following result.
\begin{proposition}
The block-counting process of the drastic bottleneck coalescent has the following transition rates
\begin{equation*}
q_{ij} \ = \ a \binom{i}{2} \mathds{1}_{\{j = i-1\}}  +  \eta \sum_{k\ge r} F^0(k) \sum_{g\ge1} \mathrm{L}(g)\p(K^{k, g-1, W^{k,i}} = j).
\end{equation*}
\label{proprate}
\end{proposition}
As in the previous section, we can define $\mathcal{\bar G}$, the infinitesimal generator of the block-counting process $\{\bar N_t\}_{t\ge0}$. For any function $h: \N \to \R$, we have
\begin{equation}
 \mathcal{\bar G}h(n) := a \binom{n}{2} \left( h(n-1) - h(n) \right)  \ + \ \eta \sum_{k\ge 1} F^0(k) \sum_{g\ge1} \mathrm{L}(g)\sum_{j =1}^{n-1}\p(K^{k,g-1,W^{k,n}} = j) \left(h(j) - h(n)\right).
\label{generatorG2}
\end{equation}
\begin{remark}
As in the previous section, we can describe the drastic bottleneck coalescent using a paintbox construction. When the bottleneck lasts for $g$ generations, it would correspond to iterating $g$ times the paintbox construction of the symmetric coalescent.
\end{remark}
 \subsection{Duality with the Wright-Fisher model with long drastic bottlenecks}
Now, we consider the Wright-Fisher model with long drastic bottlenecks from Definition \ref{longdrastic}, with two types of individuals. 
  We denote by $\{\bar X^N_g\}_ {g\in\N}$ the frequency process associated with that model. 
In the following, `$\overset{d_\lambda}{\Longrightarrow}$' denotes weak convergence in the topology induced by $d_\lambda$. 

  \begin{theorem}
Fix  $\alpha \in(0,1]$ and $\eta>0$. Let $F^0$ and $\mathrm{L}$ be two probability measures in $\N$.
Consider the sequence of processes  $\{\bar X^N\}_{N \in \N}$, such that $\bar X^N:=\{\bar X^{N}_g\}_ {g\in\Z_+}$ is the frequency process associated with the Wright-Fisher model with long drastic bottlenecks parametrized by $\alpha, \ \eta,  \ N$, $F^0$ and $\mathrm{L}$ (see Definition \ref{longdrastic}). Then, for any $T>0$, in $D[0,T]$,
$$
\{\bar X^N_{\lfloor N^{\alpha} t\rfloor}\}_{0\le t \le T} \ \overset{d_\lambda}{\underset{N\to\infty}{\Longrightarrow} } \ \{\bar X_t\}_{0\le t \le T},
$$
where $\{\bar X_t\}_{0\le t \le T}$ is the unique strong solution of the SDE
\begin{equation}
d\bar X_t \ = \ \mathds{1}_{\{\alpha = 1\}} \sqrt{\bar X_t(1-\bar X_t)}dB_t \ + \ \int_{U_0} \sum_{i=1}^k \frac{a_i}{k} \left( \mathds{1}_{\{u_i \le \bar X_{t^-} \}}  - \bar X_{t^-}\right) \bar N(dt, dk, dg, da, du),
\label{SDE2}
\end{equation}
where  $\{B_t\}_{t\geq0}$ is a standard Brownian motion and $\bar N$ is a compensated Poisson measure   on  $(0, \infty) \times U_0$ with $U_0  =  \N \times \N \times E^k\times [0,1]^{\N}$. $\bar N$ has intensity $\eta ds \otimes F^0(dk) \otimes \mathrm{L}(dg) \otimes \mathrm{A}^{k,g-1}(da) \otimes du $, where $du$ is the Lebesgue measure on $[0,1]^{\N}$ and $\mathrm{A}^{k,g-1}$ is the distribution in $E^k$ of the family sizes in a classical Wright-Fisher model with population size $k$, after $g-1$ generations (as defined in Section \ref{sect212}), i.e. $\mathrm{A}^{k,g-1}$ is the distribution of a vector whose $i$-{th} coordinate corresponds to the number of descendants of  individual $i$ after $g-1$ generations  in a classical Wright-Fisher model with population size $k$.
\label{SDEs2}
\end{theorem}

\begin{remark}
This result implies that $\{\bar X^N_{\lfloor N^{\alpha} t\rfloor}\}_{0\le t \le T}$ converges to $ \{\bar X_t\}_{0\le t \le T}$ in measure on the Skorokhod space. 
As already mentioned in the introduction (and illustrated in Figure \ref{nonmarkovian}), it is impossible to have convergence in the $J_1$ or $M_1$ Skorokhod topology. 
\end{remark}

\begin{figure}
\begin{center}
\subfigure[]{\includegraphics[width=8cm]{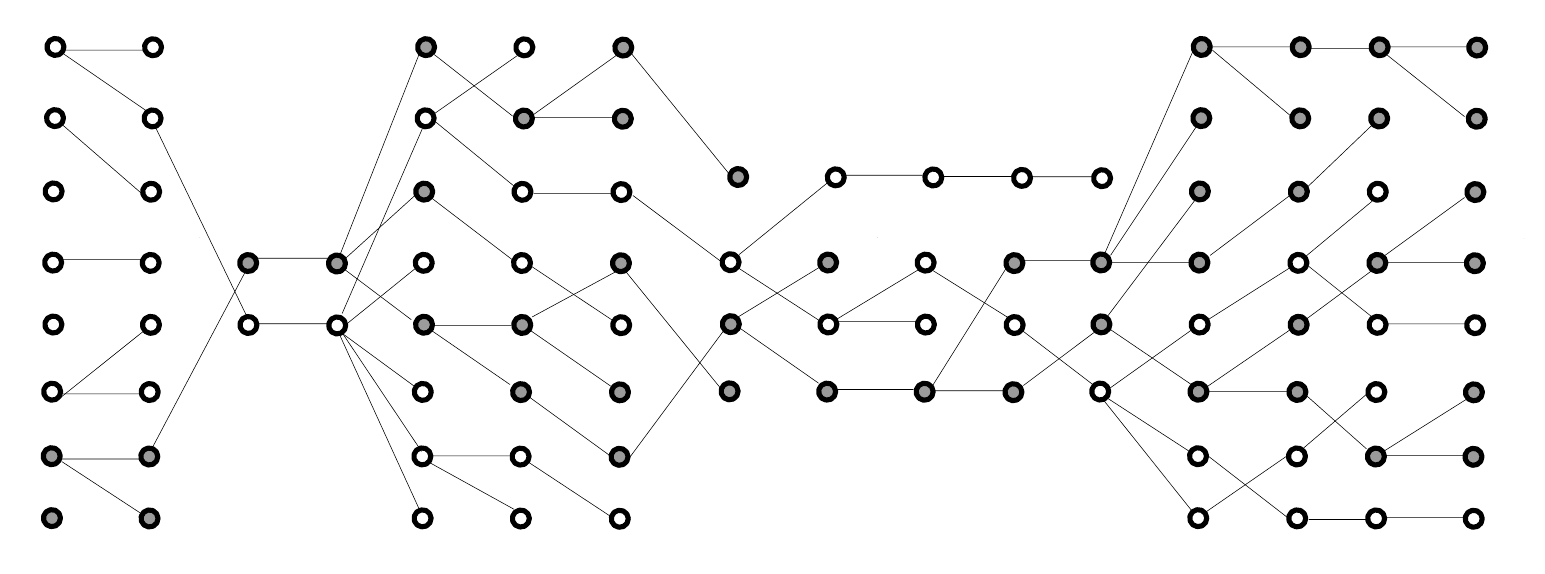}} \quad
\subfigure[]{\includegraphics[width=4.5cm]{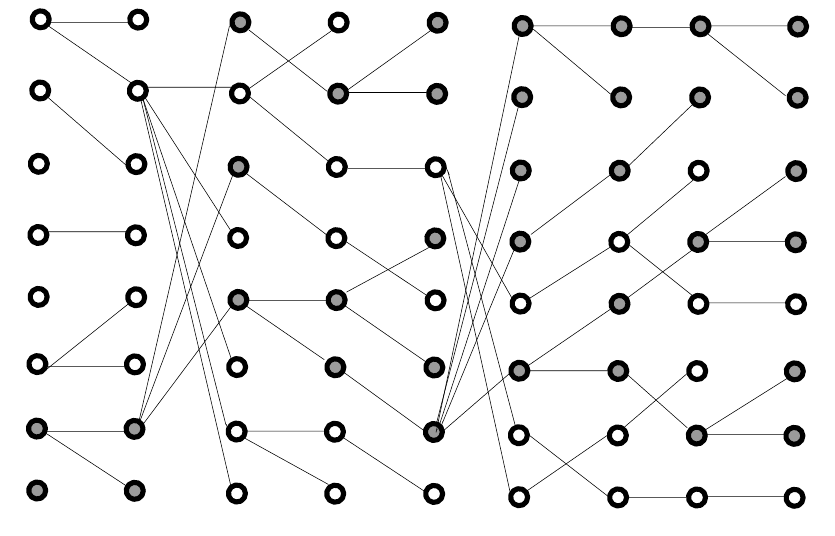}} \\
\subfigure[]{\includegraphics[width=8cm]{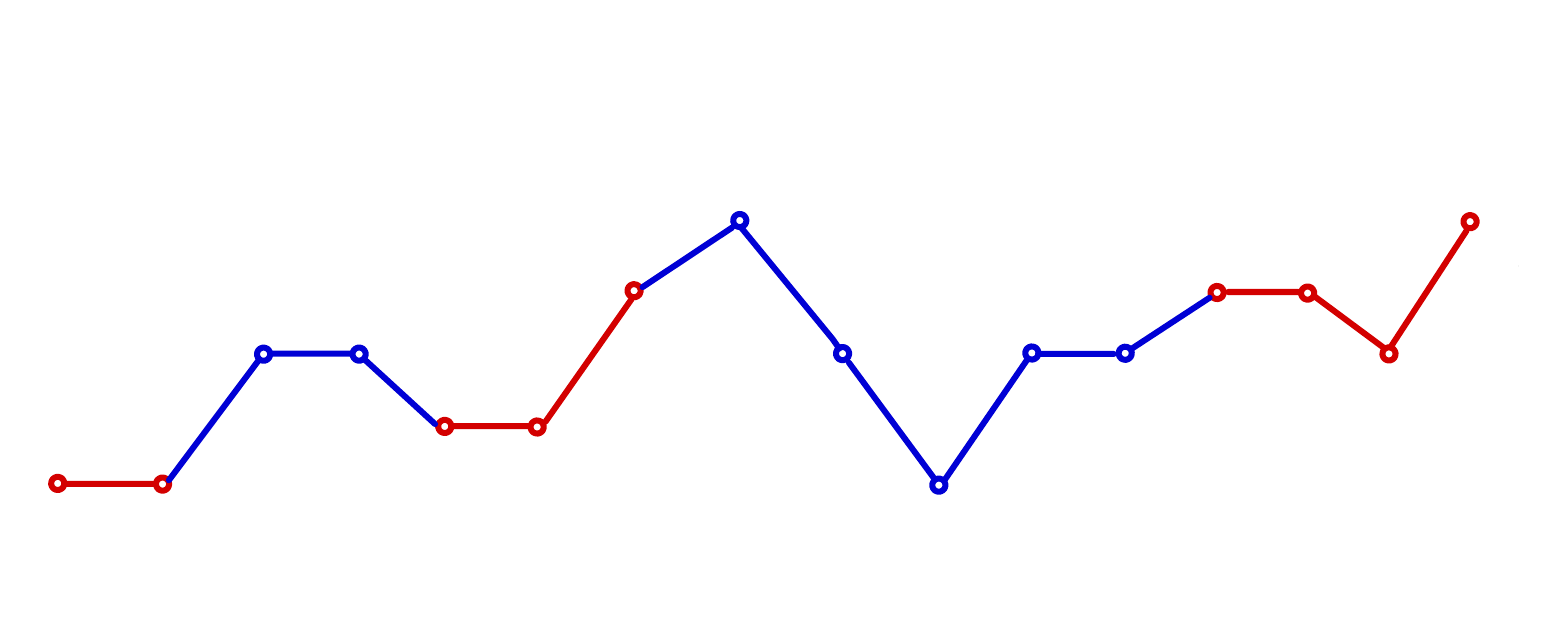}} \quad
\subfigure[]{\includegraphics[width=4.5cm]{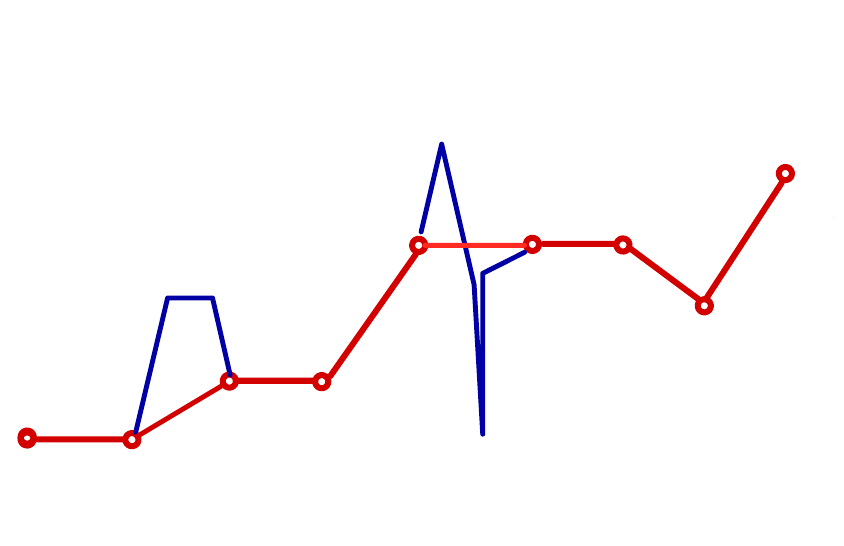}}
\caption{
(a) represents a realization of the Wright-Fisher model with long drastic bottlenecks in which the \textit{usual} population size is $N=8$. We observe that $F_1=2$, $s_{1,8}=2$ and $l_{1,8}=2$, meaning that in generation $2$ a bottleneck that reduces the population to $2$ individuals, starts and last for $2$ generations. Similarly, $F_2=4$, $s_{2,8}=3$, $l_{2,8}=5$ and $s_{3,8}\ge4$. In (c), the frequency process (associated with (a)) is colored according to the population size: red outside the bottlenecks and blue during the bottlenecks.
In (b) we observe the result of collapsing the bottlenecks in the Wright-Fisher model with long drastic bottlenecks represented in (a). We observe that a lot of the fluctuations of the frequency process are lost. In (d), the red line is the frequency process and the blue excursions, that were present in (c), are lost due to the collapsing of the bottlenecks. These blue excursions (sparks) make it impossible to have convergence in $J_1$ or $M_1$ to the diffusion with jumps  that is the solution of \eqref{SDE2}.
 }
\label{nonmarkovian}
\end{center}
\end{figure}

To understand why the jump term of \eqref{SDE2} takes this form, we need to think about what happens during a bottleneck of size $k$ and length $g$. 
When the bottleneck begins, only $k$ individuals of the infinite population survive. The term `$\mathds{1}_{\{u_i \le \bar X_{t^-}\} }$' comes from the fact that individual $i$ is of type 1 with probability $\bar X_{t^-}$ and of type $0$ with probability $1-\bar X_{t^-}$. This generation corresponds to time $0$ for the bottleneck.
Then the bottleneck lasts for another $g-1$ generations and individual $i$ has $a_i$ descendants which are all of the same type as her. See Figure \ref{figlongbottleneck} for an illustration of this jump term.

\begin{figure}
\begin{center}
\includegraphics[width=12cm]{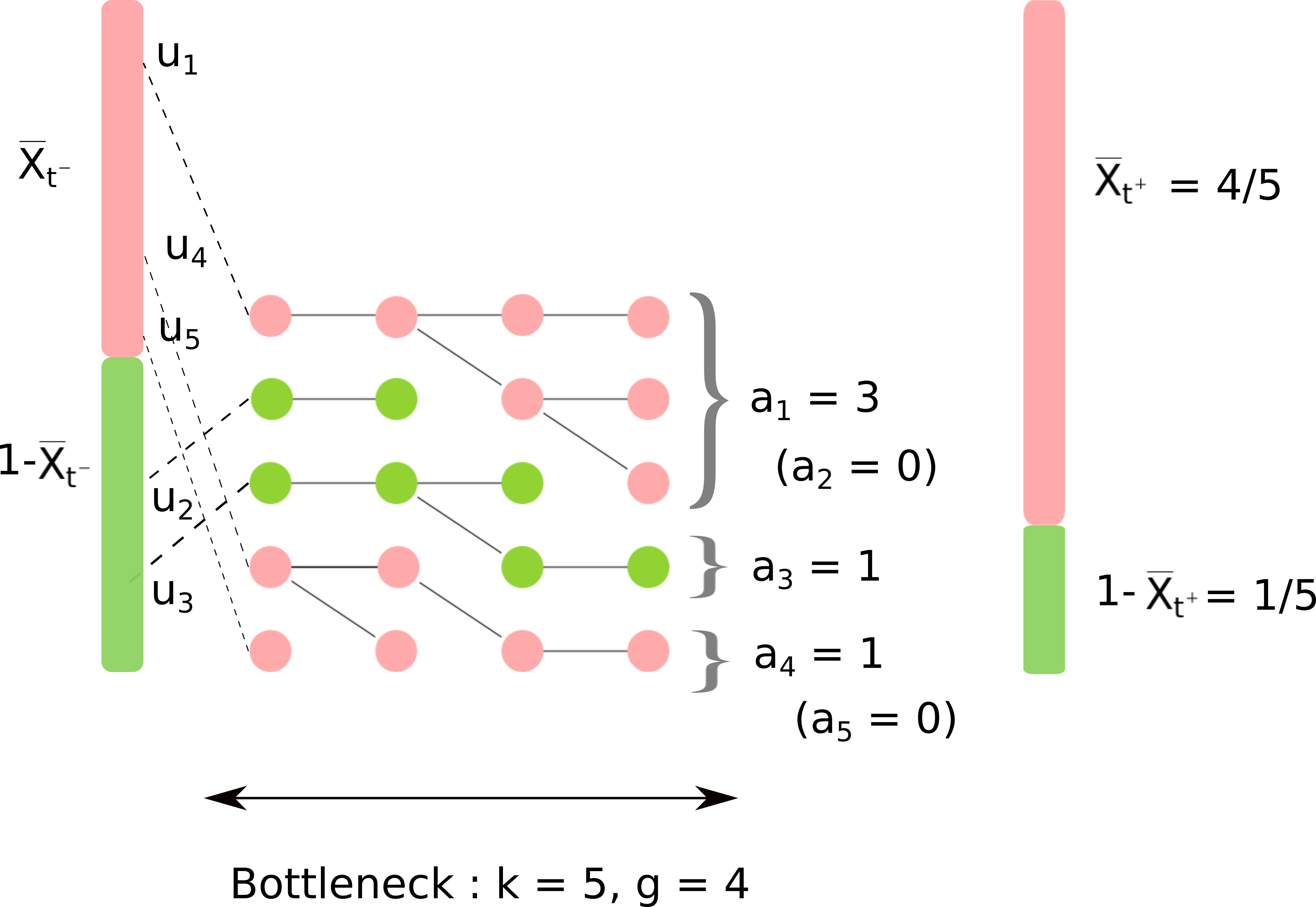}
\caption{The jump part of \eqref{SDE2} illustrated. Type 1 individuals are represented in pink and type 0 individuals in green.}
\label{figlongbottleneck}
\end{center}
\end{figure}

Again, before proving Theorem \ref{SDEs2} we shall make sure that a solution to Equation \eqref{SDE2} exists.  
\begin{lemma}\label{lemmados}
For any probability measures $F^0$ and $\mathrm{L}$ in $\N$ and any $\alpha \in (0,1], \ \eta>0$, there exists a unique strong solution to the SDE \eqref{SDE2}.
\label{lemmaexistence2}
\end{lemma}

\begin{remark}
To prove this result, one could rewrite the stochastic differential equation \eqref{SDE2} in terms of a measure $\Xi$ on $\Delta$ that would depend on $F^0$, $\mathrm{L}$ and $\mathrm{A}^{k,g-1}$ (in fact for $k\in\N$ and $(a_1, \dots, a_k)$ drawn from the distribution $\mathrm{A}^{k,g-1}$,  the re-ordering of $(a_1/k, \dots, a_k/k)$ is an element of $\Delta^*$). Then  Lemma \ref{lemmados} would follow from Lemma 3.6 in \cite{GS} (and Theorem \ref{samplingduality} would follow form Proposition 3.8 in the same reference).
However, we have decided not to modify \eqref{SDE2} and we give a proof that is more instructive, as it is connected to the parameters of the Wright-Fisher model with long drastic bottlenecks and sheds light on the connection between this Wright-Fisher model and the drastic bottleneck coalescent.
\end{remark}

\begin{proof}
We use Theorem 5.1 in \cite{Li}. In particular, we need to verify conditions (3.a), (3.b) and (5.a)  of that paper. Condition (3.a) is trivial, as in our case the drift coefficient is equal to 0. To prove condition (3.b), we have to prove that there exists a constant $K$ such that for every $x,y \in [0,1]$,
\begin{align}
& \int_{\N} \int_{\N} F^0(k) \mathrm{L}(g) \E\left  [  \left ( \sum_{i=1}^k \frac{a_i}{k} \left(  \mathds{1}_{\{u_i \le x\}} - x  -   \mathds{1}_{\{u_i \le y\}} + y \right)\right)^2 \right ]  \nonumber \\ 
&+  \mid \sqrt{x(1-x)} - \sqrt{y(1-y)} \mid   \ \le \ K \mid x- y\mid.
\label{3bPuLi}
\end{align}
First, we use the fact that 
\begin{align*}
\mid \sqrt{x(1-x)} - \sqrt{y(1-y)} \mid  \le 4 \mid x - y \mid,
\end{align*}
see for example claim (26) in \cite{GS}. 
Second,  without lost of generality we assume that $x>y$ and we have
\begin{align*}
\int_{\N} \int_{\N}F^0(k)\mathrm{L}(g) \E\left  [  \left ( \sum_{i=1}^k \frac{a_i}{k} \left(  \mathds{1}_{\{u_i \le x\}} - x  -   \mathds{1}_{\{u_i \le y\}} + y \right)\right)^2 \right ]   \\
=  \int_{\N}\int_{\N} F^0(k)\mathrm{L}(g)  \E\left [  \sum_{i=1}^k \frac{a_i^2}{k^2} \left((B_i^x - B_i^y) - (x-y)  \right)^2 \right ], 
\end{align*}
where the $B_i^x$'s and the $B_i^y$'s are (dependent) Bernoulli random variables of parameter $x$ and $y$ respectively. Using the fact that $(B_i^x - B_i^y)$ is Bernoulli of parameter $x-y$ we have: 
\begin{align*}
& \int_{\N}\int_{\N} F^0(k)\mathrm{L}(g)  \E\left [  \sum_{i=1}^k \frac{a_i^2}{k^2} \left((B_i^x - B_i^y) - (x-y)  \right)^2 \right ]  \\
&\le   ((x-y)(1-(x-y))  \  \int_{\N}\int_{\N} F^0(k) \mathrm{L}(g) \E\left(\sum_{i=1}^k \frac{a_i^2}{k^2}\right)   \\
&\le  (x-y), 
\end{align*}
where we used the facts that $\frac{a_i}{k}\le 1$, $\sum_{i=1}^k {a_i} = k$ and $F^0$ and $\mathrm{L}$ are probability measures.  This proves claim \eqref{3bPuLi}. 
Finally, condition (5.a) in \cite{Li} is verified because, using similar arguments as before,
\begin{align}
& \int_{\N} \int_{\N} F^0(k) \mathrm{L}(g) \E\left  [  \left ( \sum_{i=1}^k \frac{a_i}{k} \left(  \mathds{1}_{\{u_i \le x\}} - x \right)\right)^2 \right ]  + x(1-x)   \ \le \ 2.
\label{5aPuLi}
\end{align}
This implies that we can apply Theorem 5.1 in  \cite{Li} and conclude the existence and uniqueness of strong solution of \eqref{SDE2}.
 \end{proof}
 \begin{remark}
 The fact that $F^0$ and $\mathrm{L}$ are probability measures guarantees the existence of a unique strong solution to \eqref{SDE2}, but a sufficient condition is that $$\int_{\N}\int_{\N} F^0(k) \mathrm{L}(g) \E\left(\sum_{i=1}^k \frac{a_i^2}{k^2}\right) \ < \ \infty,$$
 where $(a_1, \dots, a_k)$ is distributed as $\mathrm{A}^{k,g-1}$.
 \end{remark}
 
We are now ready to prove Theorem \ref{SDEs2}. 
The strategy of the proof is the following. We want to prove the convergence of the sequence of processes $\{\bar X^N\}_{N\in\mathbb N}$, which are not Markovian, to a diffusion with jumps. 
In Step 1 we will construct an auxiliary process $V^N$ that corresponds to collapsing each bottleneck into one single time step into $\bar X^N$ and that is Markovian. We will prove that $d_\lambda(\{\bar X^N_{\lfloor N^\alpha t\rfloor}\}_{0\le t \le T}, \{V^N_{\lfloor N^\alpha t\rfloor}\}_{0\le t \le T})\rightarrow 0$ in probability.
In Step 2, we prove, using standard techniques for Markov processes, that $\{V^N_{\lfloor N^\alpha t\rfloor}\}_{0\le t \le T}$ converges weakly to $\{\bar X_t\}_{0\le t \le T}$, as processes in $D[0,T]$ endowed with Skorokhod's $J_1$ topology. 
Combining the two steps, the conclusion is straightforward. 

This method can be generalized to any case in which convergence in the Skorokhod topology is prevented by strong fluctuations that happen in a set of times that has Lebesgue measure 0 in the limit.

\begin{proof}[Proof of Theorem \ref{SDEs2}] 
\textbf{Step 1.}
We say that a generation is in a bottleneck and write $g\in B^N$ if for some $m$, $\underline{t}_m:=\sum_{i=1}^{m-1}(s_{i, N} + l_{i,N}) + s_{m,N} < g  \le  \bar{t}_m:=\sum_{i=1}^{m}(s_{i, N} + l_{i,N})$. Let $g_i$ be the $i$-th generation that is not in a bottleneck. Define 
$$
V^N_i=\bar X^N_{g_i}.
$$  
Note that while $\bar X^N$ is not a Markov process (as one needs to know if there is a bottleneck or not to calculate the transition probabilities), $V^N$ is a Markov process, and it is constructed from $\bar X^N$ simply by collapsing the entire bottleneck into one step. Now, consider the random $R^N_g$-measurable projection $\pi^N:\R_+\mapsto \N$ defined as $\pi^N(t)=\max\{g\in \N: g<t, g\in (B^N)^c\}$ and we define the process $\{Z^N_t\}_{t\geq0}$ by
$$
Z^N_t=\bar X^N_{\pi^N(t)}.
$$
The difference between $V^N_{\lfloor t\rfloor}$ and $Z^N_t$ is that $V^N_{\lfloor t\rfloor}$ can move every unit of time, while $Z^N_t$ stays still during the duration of a bottleneck (see Figure \ref{trajectoriesXVZ} for an illustration).

Taking $f\in \mathcal{F}$ to be the linear function having slope 1 whenever $\lfloor t \rfloor$ is not in a bottleneck and having slope $({\bar t_m-\underline{t}_m})$ whenever $\lfloor t \rfloor$ is inside the $m$-th bottleneck, we conclude that $||V^N_{\lfloor N^\alpha t\rfloor}-Z^N_{f(\lfloor N^\alpha t\rfloor)}||=0$ (see Figure \ref{trajectoriesXVZ}) and thus 
$$d_\lambda(\{V^N_{\lfloor N^\alpha t\rfloor}\}_{0\le t \le T}, \{Z^N_{\lfloor N^\alpha t\rfloor}\}_{0\le t \le T})\leq ||Id-f||\leq \frac{\sum_{m=1}^\infty (\bar{t}_m-\underline{t}_m)\mathds{1}_{\{\bar{t}_m\leq TN^\alpha\}}}{N^\alpha}.$$ Observing that the latter converges to 0 in probability when $N\rightarrow \infty$, we conclude that 
$$
d_\lambda(\{V^N_{\lfloor N^\alpha t\rfloor}\}_{0\le t \le T},\{Z^N_{\lfloor N^\alpha t\rfloor}\}_{0\le t \le T})\rightarrow 0 \textrm{ in probability}.
$$
Now, observe that taking $A^N=\cup_{m=1}^{i_T^N}[\underline{t}_m,\bar{t}_m)$ where $i_T^N=\sup\{m:\underline{t}_m<TN^\alpha\}$ we have
$$
||\mathds{1}_{(A^N)^c}(Z^N_{\lfloor N^\alpha t\rfloor}-\bar X^N_{\lfloor N^\alpha t\rfloor})||=0  \textrm{ in probability}.
$$
and thus 
$$
d_\lambda(\{\bar X^N_{\lfloor N^\alpha t\rfloor}\}_{0\le t \le T},\{Z^N_{\lfloor N^\alpha t\rfloor}\}_{0\le t \le T})\rightarrow 0 \textrm{ in probability},
$$
which completes this step.

\begin{figure}
\begin{center}
\includegraphics[width=7cm]{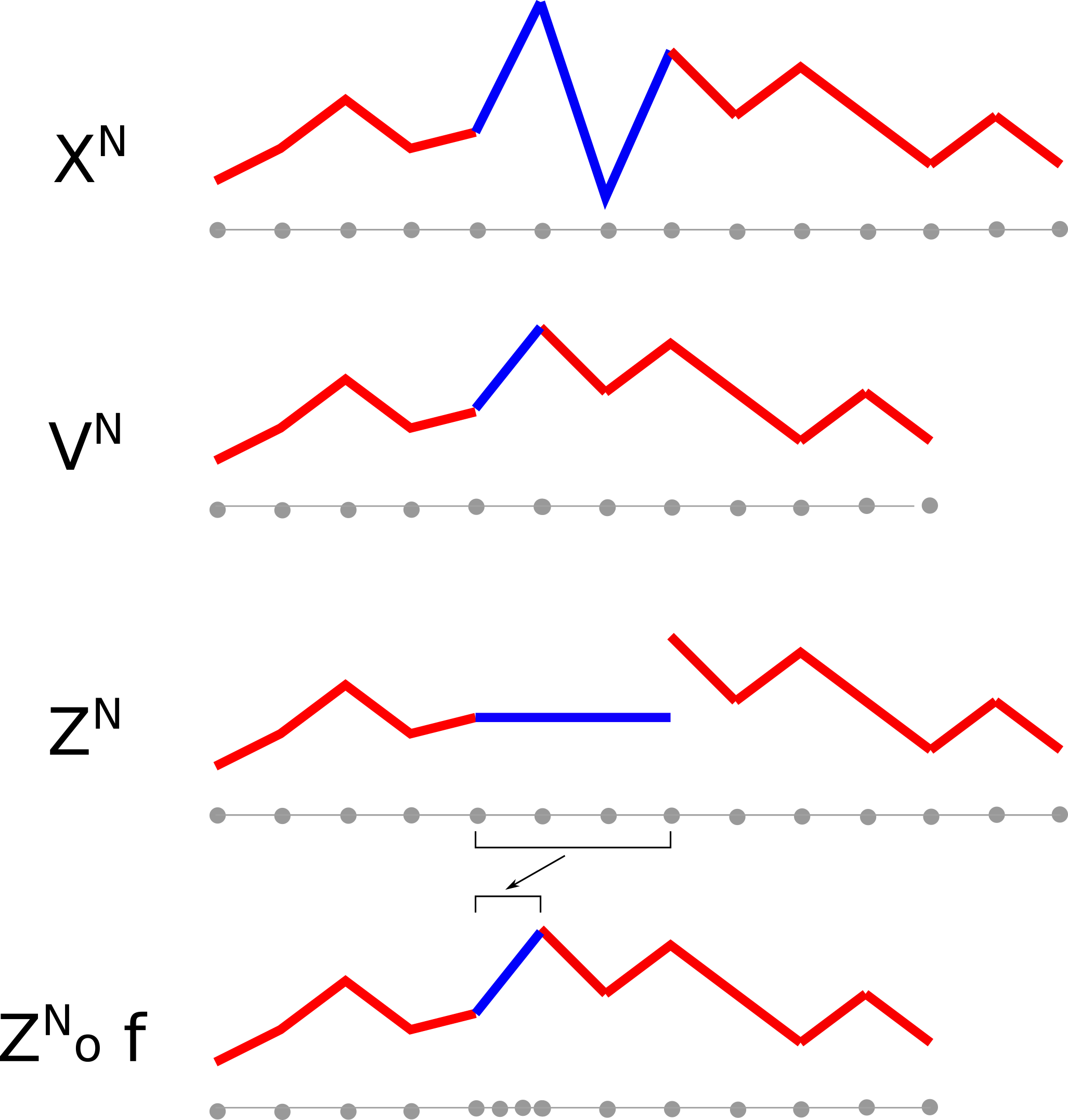}
\caption{The curves represent realizations of $X^N$, $V^N$, $Z^N$ and $Z^N\circ {f}$ respectively. The first curve  is colored according to the population size: red outside the bottlenecks and blue during the bottlenecks. }
\label{trajectoriesXVZ}
\end{center}
\end{figure}

\textbf{Step 2.} As in the proof of Theorem \ref{SDEs}, the idea is to prove the convergence of the generator of $\{V^N_{\lfloor N^{\alpha}t \rfloor}\}_{t\geq0}$, to the generator of $\{\bar X_t\}_{t\geq0}$. Provided this claim is true, we can use Theorem 19.25 and 19.28 of \cite{Kallemberg} to prove the weak convergence of $\{V^N_{\lfloor N^{\alpha} t \rfloor}\}_{0\le t \le T}$ towards $\{\bar X_t\}_{0\le t \le T}$.

From Lemma \ref{lemmaexistence2}, $\{\bar X_t\}_{t\geq0}$ exists and has generator
 $\mathcal{\bar A}$. Its domain contains twice differentiable functions and  for a function $f \in C^2[0,1]$ and $x \in [0,1]$, we have
\begin{equation}
\mathcal{\bar A}f(x) \ = \ \mathds{1}_{\{\alpha = 1\}}\frac12x(1-x)f''(x) \ + \ \eta\sum_{k\ge 1} F^0(k) \sum_{g \ge1} \mathrm{L}(g)  \mathcal{\bar A}^{k,g}f(x),
\label{generatorA2}\end{equation}
with
$$ \mathcal{\bar A}^{k,g}f(x) = \sum_{i =0}^k \p(\bar Y^{k, g,x} = i/k) \left( f(i/k) - f(x) \right),$$
where $\bar Y^{k,g, x}$  is a random variable such that $ \bar Y^{k, g, x} = \sum_{i=1}^k \frac{a_i}{k}  \mathds{1}_{\{u_i \le x\} } $, where $(a_1, \dots, a_k)$ is distributed as $\mathrm{A}^{k,g-1}$ and $(u_1,\ldots, u_k)$ is uniformly distributed in $[0,1]^k$. This generator can be interpreted in the same way as the jump part of \eqref{SDE2}, see Figure \ref{figlongbottleneck}.

The discrete generator $\mathcal{\bar A}^N$ of $\{ V^N_{\lfloor N^{\alpha} t \rfloor}\}_{t\geq0}$ (defined as in \eqref{discretegenerator}), applied to a function $f\in C^2[0,1]$ in $x \in [0,1]$ can be written as
\begin{align*}
&\mathcal{\bar A}^Nf(x)  = \ N^{\alpha}  (1 - \frac{\eta}{N^{\alpha}})  \E \left( f\left(\frac{\sum_{i=1}^NB^x_i}{N}\right) - f(x)\right)\\ 
&+ \eta  \sum_{k\ge1} F^0(k) \sum_{g \ge1} \p(l_{1,n}=g)  \sum_{i =0}^k \p\left(\bar Y^{\min(N,k),g, x} =\frac{ i}{\min(N,k)}\right) \left( f(\frac{i}{\min(N,k)}) - f(x) \right).
\end{align*}
The first term corresponds to the generator of a classical Wright-Fisher model (outside the bottlenecks). As already mentioned, it is well-known that when $\alpha = 1$, this term converges when $N\to \infty$ to $\frac12x(1-x)f''(x)$, which is the generator of the Wright-Fisher diffusion. When $\alpha <1$ this term becomes of order $N^{\alpha-1}$ and therefore converges to $0$. 
The second term corresponds to what happens during a bottleneck (and again, it can be interpreted using Figure \ref{figlongbottleneck}). Recall that $\mathrm{L}/N^\alpha \to 0 $ in distribution, i.e. in the new time scale the bottlenecks are instantaneous.
As $l_{1,N}$ converges in distribution to $\mathrm{L}$, the second term in the generator converges when $N\to \infty$ to the second term of $\mathcal{\bar A}$ (see \eqref{generatorA2}).
Combining these two results, we have $\mathcal{\bar A}^Nf \to \mathcal{\bar A}f$
uniformly. This implies that $\{ V^N_{\lfloor N^{\alpha} t \rfloor}\}_{0\le t \le T}$ converges weakly in the Skorokhod $J_1$ topology to $\{ \bar X_t\}_{0\le t \le T}$. Since convergence in $J_1$ implies convergenec in $d_\lambda$ we have the desired result.
\end{proof}

We fix $\alpha \in (0,1],  \ \eta >0$ and $F^0$ and $\mathrm{L}$ two probability measures in $\N$.
Let us consider $\{\bar N_t\}_{t\geq 0}$, the block-counting process of the drastic bottleneck  coalescent characterized by $\alpha$, $\eta$, $F^0$ and $\mathrm{L}$.
 As in the previous section, we are going to prove a duality relation between this block-counting process and $\{\bar X_t\}_{t\geq 0}$, the unique strong solution of \eqref{SDE2}  (with the same parameters). 
\begin{theorem}
For every $x \in [0,1], \ n \in \N$, we have
$$\E(\bar X_t^n \vert \bar X_0 = x) \ = \ \E (x^{\bar N_t} \vert \bar N_0 = n).$$
\label{samplingduality}
\end{theorem}
\begin{proof}
We start by recalling a moment duality between the frequency process of a  classical Wright-Fisher model with population size $k$, started at $x$, denoted by $\{Y^{k,g,x}\}_{g\in \N}$, and the number of blocks in the associated ancestry process of a sample of size $n$, $\{K^{k,g,n}\}_{g \in \N}$, which was established by M\"ohle (Proposition 3.5 in \cite{mohle1}). We consider the function $h$ on $[0,1]\times \N$ such that $h(x,n) = x^n$. We have
$$\E(h( Y^{k,g,x}, n) ) \ = \ \E (h(x, K^{k,g,n}) )$$
i.e.,
\begin{equation}
 \sum_{i = 0}^k \p( Y^{k,g,x} = i/k) h(i/k,n) \ = \ \sum_{b =1}^n \p(K^{k,g,n} = b) h(x,b).
\label{sampling1}
\end{equation}

We start by considering the case $\alpha < 1$.  Using the Markov property of the Wright-Fisher model with bottlenecks (i.e. the fact that the $u_i$'s are independent from the $a_i$'s) we can rewrite $ \bar Y^{k,g,x} $ so that
for every $n \in \N$, the generator $\mathcal{ \bar A}$ applied to $h$ (seen as a function of $x$) is
\begin{align*}
\mathcal{ \bar A}h(x,n) & = \eta\sum_{k\ge1} F^0(k) \sum_{g \ge1} \mathrm{L}(g)\sum_{j = 0}^k \sum_{y=0}^k \p(\sum_{i=1}^k \mathds{1}_{\{u_i \le x\}} = y)\p(Y^{k,g-1,y/k} = \frac{j}{k}) (h(\frac{j}{k},n) - h(x,n)).
\end{align*}
Using the Markov property of $\{Y^{k,g, y/k}\}_{g\in\N}$ we have that
$$\sum_{j = 0}^k \p(Y^{k,g-1,y/k} = \frac{j}{k}) = \sum_{j = 0}^k \sum_{p = 0}^k \p(Y^{k,g-2,y/k}= \frac{p}{k}) \p(Y^{k,1, p/k} = \frac{j}{k}).$$
Then, by definition of the classical Wright-Fisher model, we have that
$Y^{k,1, p/k}$ has the same distribution as $\frac1{k} \sum_{i = 1}^k B_i^{p/k}$, where the $B_i^{p/k}$'s are Bernoulli variables of parameter $p/k$, so
using exactly the same computations as in \eqref{dualitybinom1} and \eqref{dualitybinom2}, we have that
\begin{align*}
\E\left((Y^{k,1, p/k})^n \right) \ = \ \E \left( (p/k)^{W^{k,n}}\right),
\end{align*}
so, 
\begin{align*}
\sum_{j = 0}^k \p(Y^{k,g-1,y/k} = \frac{j}{k})h(\frac{j}{k}, n) & = \sum_{p = 0}^k \p(Y^{k,g-2,y/k} = \frac{p}{k}) \sum_{j = 0}^k \p(Y^{k,1, p/k} = \frac{j}{k}) h(\frac{j}{k}, n)\\
&  = \sum_{p = 0}^k \p(Y^{k,g-2,y/k} = \frac{p}{k})\sum_{m=1}^n \p(W^{k,n}= m)h(\frac{p}{k}, m)\\
&  =  \sum_{m=1}^n \p(W^{k,n}= m) \sum_{b= 1}^k \p(K^{k,g-2,m} = b) h(\frac{y}{k}, b)\\
&  =  \sum_{b = 1}^k \p(K^{k,g-2,W^{k,n}} = b) h(\frac{y}{k}, b),
\end{align*}
where, from the second to third line we used the duality relation \eqref{sampling1}. Replacing into the expression of the generator $\mathcal{ \bar A}$, we have that
\begin{align*}
\mathcal{ \bar A}h(x,n) & = \eta\sum_{k\ge1} F^0(k) \sum_{g \ge1} \mathrm{L}(g) \sum_{y=0}^k \p(\sum_{i=1}^k \mathds{1}_{\{u_i \le x\}} = y)\sum_{b = 1}^k \p(K^{k,g-2,W^{k,n}} = b) (h(\frac{y}{k},b) - h(x,n) ).
\end{align*}
Recall that, if $x \in \{0, 1/k, \dots, 1\}$, the distribution of $\frac1{k}\sum_{i=1}^k \mathds{1}_{\{u_i \le x\}}$ is exactly the distribution of $Y^{k,1,x}$, so we can apply \eqref{sampling1} to $Y^{k,1,x}$ and $K^{k,1,b}$. However, we need to prove that a similar relation exists when $x \in [0,1]\setminus\{0, 1/k, \dots, 1\}$.
In fact, using again the same computations as in \eqref{dualitybinom1} and \eqref{dualitybinom2}, we have that
\begin{align*}
\E\left( (\frac1{k}\sum_{i=1}^k \mathds{1}_{\{u_i \le x\}})^b \right) \ = \ \E \left( x^{W^{k,b}}\right),
\end{align*}
and we let the reader convince herself that $W^{k,b}$ has the same distribution as $K^{k,1,b}$.
Finally, 
\begin{align*}
\mathcal{ \bar A}h(x,n) & = \eta\sum_{k\ge1} F^0(k) \sum_{g \ge1} \mathrm{L}(g) \sum_{b = 1}^k \p(K^{k,g-2,W^{k,n}} = b) \sum_{a=1}^k \p(K^{k,1,b} = a) (h(x,a) - h(x,n))\\
& = \eta\sum_{k\ge1} F^0(k) \sum_{g \ge1}  \sum_{a=1}^k \p(K^{k,g-1,W^{k,n}} = a) \left(h(x,a) - h(x,n)\right) \ = \ \mathcal{ \bar G}h(x,n),
\end{align*}
where in the last line we used the Markov property of $\{K^{k,g,n}\}_{g \in \N}$ and $\mathcal{ \bar G}$ is the generator of $\{\bar N_t\}_{t\geq 0}$ (defined in \ref{generatorG2}) applied to $h$, seen as a function of $n$.

If $\alpha = 1$  we have
\begin {align*}
\mathcal{\bar A}h(x,n) \ = \ \mathcal{A}_1h(x,n) +\mathcal{\bar A}_2h(x,n)
\end{align*}
where $\mathcal{\bar A}_2$ is the infinitesimal generator of $\{\bar X_t\}_{t\geq0}$  for the case $\alpha <1$ and $\mathcal{A}_1$ is the generator of the Wright-Fisher diffusion. 
The proof follows by using the moment duality between the Wright-Fisher diffusion and the Kingman coalescent, i.e. $$ \ \mathcal{A}_1h(x,n) =  \ \binom{n}{2} \left( h(x,n-1) - h(x,n) \right).$$
This, combined with the case $\alpha<1$ completes the proof.
\end{proof}

\section{Coalescents with soft bottlenecks}\label{soft}
\subsection{Kingman coalescents with continuous time rescaling}
In this section we consider bottlenecks that are soft and short. 
More precisely, we consider a Wright-Fisher model with bottlenecks where the sequence of population sizes $\{R^N_g\}_{g \in \Z_+}$ is such that
$$R^N_g = \lfloor N r_g \rfloor + 1$$
 where $\{r_g\}_{g \in\Z_+}$ is a sequence of i.i.d. random variables with a certain law on $[0,1)$ that we denote by $R$. We review here some examples, in which the genealogy is a time rescaled Kingman coalescent. 
The proofs are based on M\"ohle's theorem  \cite{mohle3} that can be rewritten as follows. Let us denote by $\{\Pi^N_t\}_{t\geq0}$ the ancestral process describing the genealogy of a Wright-Fisher model with bottlenecks parametrized by $N$ and $\{R^N_g\}_{g\in\N}$ and let $\{\Pi_t\}_{t\geq0}$ be the standard Kingman coalescent.
\begin{proposition}[M\"ohle's theorem for the Wright-Fisher model with bottlenecks]
Consider 
$$
C_N=\sum_{i=1}^N\frac{1}{i}\p(R^N_1=i)  
\textrm{ and }
D_N=\sum_{i=1}^N\frac{1}{i^2}\p(R^N_1=i).
$$
If
$$ C_N \underset{N\to \infty}{\longrightarrow} 0  \textrm{ and }
 \  D_N/C_N \underset{N\to \infty}{\longrightarrow} 0,$$
then  $\{\Pi^N_{\lfloor t/C_N\rfloor}\}_{t\geq0}$ converges to  $\{\Pi_t\}_{t\geq0}$  in the sense of finite dimensional distributions.\label{mohlebottleneck}
\end{proposition}

\textbf{Example 1.}
We start by considering the case where there exists $\epsilon>0$ such that $\p(R>\epsilon)=1$, then the genealogy of the model converges to a constant time rescaling of the Kingman coalescent.
To prove it we use M\"ohle's theorem. In fact, 
$$
C_N=\sum_{i=1}^N\frac{1}{i}\p(R^N_1=i)=\sum_{i=\lfloor \epsilon N \rfloor}^N\frac{1}{i}\p(R \in [\frac{i-1}{N}, \frac{i}{N}))\in[\frac{1}{N},\frac{1}{\epsilon N}]
$$
and
$$
D_N=\sum_{i=1}^N\frac{1}{i^2}\p(R^N_1=i)=\sum_{i=\lfloor \epsilon N \rfloor}^N\frac{1}{i^2}\p(R \in [\frac{i-1}{N}, \frac{i}{N}))\in[\frac{1}{N^2},\frac{1}{(\epsilon N)^2}].
$$
So   $C_N\rightarrow 0$,  $D_N/C_N\rightarrow 0$ and we can apply Proposition \ref{mohlebottleneck}.

\medskip 

\textbf{Example 2.}
We assume that $R$ is a uniform random variable in $[0,1]$, i.e. for $g \in \Z_+$ $R_g^N$ is uniformly chosen in $\{1, \dots, N\}$.
Informally, this means that going backward in time, one has to wait, on average, for $N$ generations until the population size is reduced to only 1 individual. Any ancestral lineages present at that time need to coalesce into one. However, the limiting genealogy is still a Kingman coalescent. In fact, 
$$C_N \ = \ \frac{1}{N}\sum_{i=1}^N\frac{1}{i}\sim \frac{\log N}{N}$$
and
$$
D_N \  = \ \frac{1}{N}\sum_{i=1}^N\frac{1}{i^2}=O(1/N).
$$So   $C_N\rightarrow 0$,  $D_N/C_N\rightarrow 0$
 and Proposition \ref{mohlebottleneck} implies that
$\{\Pi^N_{\lfloor t N/\log N\rfloor}\}_{t \geq0} \rightarrow \{\Pi_{t}\}_{t \geq0}$  in the sense of finite dimensional distributions.

\subsection{Subordinated Kingman coalescents}
\label{sub-kingman}
In this section we consider bottlenecks that are soft but that last for several generations. As in Section \ref{sect212}, we are going to assume that, in the limiting model, the times between two bottlenecks are exponentially distributed. But this time we are going to assume that the bottlenecks are soft i.e. $b_{i,N} \to 0 $ but $Nb_{i,N} \to \infty$.
 We present an example inspired by Birkner et al.  \cite{BBMST}, in which the limiting genealogy is a subordinated Kingman coalescent. 

\begin{definition}[Wright-Fisher model with long soft bottlenecks]
Fix  $\alpha \in(0,1]$, $\eta >0$, $N \in \N$ and  $\mathrm{L}_{\gamma}$ a  probability measure on $\R_+$.    
Let $\{b_{i,N}\}_{i\in \N}$, $\{l_{i,N}\}_{i\in \N}$ and $\{s_{i,N}\}_{i\in \N}$ be three sequences of independent positive random variables. 
For any $ i \in \N$, assume that  $b_{i,N} \to 0$ in distribution, that $l_{i,N}/(N b_{i,N}) \to \gamma_i$ in distribution, where  $\gamma_i$ is a random variable with law $\mathrm{L}_\gamma$, and that
$s_{i,N}$ follows a geometric distribution of parameter $\eta/N^\alpha$.
In the Wright-Fisher model with long drastic bottlenecks, the sequence of population sizes $\{R^N_g\}_{g \in \Z_+}$ is given by
 $$
R^N_g \ = \ \left\{ 
\begin{array}{ll}
b_{m,N} N  \textrm{ if } \sum_{i=1}^{m-1}(s_{i, N} + l_{i,N}) + s_{m,N} < g  \le  \sum_{i=1}^{m}(s_{i, N} + l_{i,N}) \\
N \textrm{ otherwise}
\end{array} 
\right..
$$
\label{longsoft}
\end{definition}

As suggested by  Birkner et al. \cite{BBMST}, we show in Section \ref{duality2} that when $N\to\infty$ and time is rescaled by $N^\alpha$,  the genealogy is described by $\{\Pi_{S_t}\}_{t\geq0}$ where $\{S_t\}_{t\geq0}$ is a subordinator (a compound Poisson process with L\'evy measure $\eta {L}_\gamma$ and drift 1).
Proposition 6.3 in \cite{BBMST} states that this subordinated Kingman coalescent is in fact a $\Xi$-coalescent with characterizing measure  
$$\Xi =  a \delta_{(0,0, \dots)}+ \Xi_{KS} \nonumber \\
$$
where
\begin{equation}
\Xi_{KS}(d\zeta) =  (\zeta, \zeta) \int_{(0,\infty)} \sum_{j=1}^\infty \p(K_{\sigma} = j) \eta \mathrm{L}_\gamma(d\sigma)D_j(d\zeta),
\label{xiks}
\end{equation}
and $K_t$ is the number of lineages at time $t>0$ in the standard Kingman coalescent starting with $K_0 = \infty$ and $D_j$ is the law of the re-ordering of a ($j$-dimensional) Dirichlet $(1,\dots,1)$ random vector according to decreasing size. 
This result can be interpreted as follows: the simultaneous multiple collisions part in the measure $\Xi$ corresponds to the way the lineages coalesce during the bottlenecks. As the population size during the bottleneck still goes to infinity, its evolution is still given by a Kingman coalescent and it lasts for a time distributed according to $\mathrm L_\gamma$. 
The frequencies of the remaining blocks have a Dirichlet distribution (see for example  Corollary 2.1 in \cite{Berestycki}).

\begin{remark}
In an informal sense, this model can be seen as a limiting scenario for the model presented in Section \ref{sect212} when the population size during the bottleneck goes to infinity. In fact, in the model from Definition \ref{longsoft} we have
$$
\lambda_{b, (k_1,  \dots, k_r)}
=\mathds{1}_{\{r = b-1, k_1=2\}}  +  \mathcal{N}(n, (k_1,\dots, k_r))^{-1} \eta  \int_{(0,\infty)} {L}_{\gamma}(d\sigma)   \p(X^{\sigma, b} =(k_1, \dots, k_r))
$$
where $X^{\sigma, b}$ is the vector of the sizes of the blocks of a Kingman coalescent at time $\sigma$, starting with $b$ blocks (without taking into account the ordering).
The latter can be understood as a $k=\infty$ version of Proposition \ref{proprate}.
\end{remark}

\subsection{Duality between the Wright-Fisher model with long soft bottlenecks and the subordinated Kingman coalescent}
 \label{duality2}
Finally, we consider the Wright-Fisher model with long soft bottlenecks from Definition \ref{longsoft}, with two types of individuals. 
  We denote by $\{\hat X^N_g\}_ {g\in\N}$ the frequency process associated with that model.

  \begin{theorem}
  Let $\mathrm{L}_{\gamma}$ be a probability measure on $\mathbb{R}_+$.
Fix  $\alpha \in(0,1]$ and $\eta>0$.  
Consider the sequence of processes  $\{\hat X^N\}_{N \in \N}$, such that $\hat X^N=\{\hat X^N_g\}_{g\in\Z_+}$ is the frequency process associated with the Wright-Fisher model with long soft bottlenecks parametrized by $\alpha$, $\eta$, $N$ and $\mathrm{L}_{\gamma}$ (see Definition \ref{longsoft}). Then, for all $T>0$, in $D[0,T]$,
$$
\{\hat X^N_{\lfloor N^{\alpha} t\rfloor}\}_{0\le t \le T} \overset{d_\lambda}{\underset{N\to\infty}{\Longrightarrow} } \{\hat X_t\}_{0\le t \le T},
$$
where $\{\hat X_t\}_{t\ge0}$ is the strong solution of the SDE
\begin{equation}
d\hat X_t \ = \ \mathds{1}_{\{\alpha = 1\}} \sqrt{\hat X_t(1-\hat X_t)}dB_t \ + \ \int_{[0,1]}\int_{[0,1]} \sum_{i\ge1} {\zeta_i} \left( \mathds{1}_{\{u_i \le \hat X_{t^-}\} }  - \hat X_{t^-}\right) \tilde N(dt, d\zeta, du),
\label{SDE3}
\end{equation}
where  $\{B_t\}_{t\ge0}$ is a standard Brownian motion and $\hat N$ is a compensated Poisson measure   on  $(0, \infty) \times \Delta \times [0,1]^{\N}$. $\hat N$ has intensity $\eta ds \otimes \frac{\Xi_{KS}(d \zeta)}{(\zeta, \zeta)} \otimes du $, where $du$ is the Lebesgue measure on $[0,1]^{\N}$  and $\Xi_{KS}$ is the characterizing measure of the subordinated Kingman coalescent, defined in \eqref{xiks}. 
\label{SDEs3}
\end{theorem}

Again, before proving Theorem \ref{SDEs3} we shall make sure that a solution to Equation \eqref{SDE3} exists.  
\begin{lemma}
For any probability measure $\mathrm{L}_{\gamma}$ in $\R_+$ and any $\alpha \in (0,1], \ \eta>0$, there exists a unique strong solution to the SDE \eqref{SDE3}. 
\label{lemmaexistence3}
\end{lemma}

\begin{proof}
See Proposition 3.4 in \cite{GS}.
 \end{proof}
We are now ready to prove Theorem \ref{SDEs3}. 

\begin{proof}[Proof of Theorem \ref{SDEs3}]
We use the same strategy as in the proof of Theorem \ref{SDEs2}. 
Again, if $g_i$ is the $i$-th generation that is not in a bottleneck, we define $$
\hat V^N_i=\hat X^N_{g_i}.
$$ 
Following Step 1 in the proof of Theorem \ref{SDEs2} we have 
$$d_\lambda(\{\hat X^N_{\lfloor N^\alpha t\rfloor}\}_{0\le t \le T}, \{\hat V^N_{\lfloor N^\alpha t\rfloor}\}_{0\le t \le T})\rightarrow 0 \textrm{ in probability}.$$
Again, we need to prove the convergence of the generator of $\{\hat V^N_{\lfloor N^{\alpha}t \rfloor}\}_{t\ge0}$, to the generator of $\{\hat X_t\}_{t\ge0}$. 

From Lemma \ref{lemmaexistence3}, $\{\hat X_t\}_{t\ge0}$ exists and has generator
 $\mathcal{\hat A}$. Its domain contains twice differentiable functions and  for a function $f \in C^2[0,1]$ and $x \in [0,1]$, we have
 \begin{align}
\mathcal{\hat A}f(x) \ =& \ \mathds{1}_{\{\alpha = 1\}}\frac12x(1-x)f''(x) \ + \ \int_{\Delta} \frac{\Xi_{KS}(d\zeta)}{(\zeta, \zeta)} \E\left(f\left(\sum_{i\ge1}\zeta_i B_i^x\right) - f(x)\right),
\label{generatorAhat}
\end{align}
where the $B_i^x$'s are Bernoulli random variables of parameter $x$ and the second term is the generator of the frequency process associated with a $\Xi_{KS}$-Fleming-Viot process, see for example formula (5.6) in \cite{BBMST}.

The discrete generator $\mathcal{\hat A}^N$ of $\{\hat V^N_{\lfloor N^{\alpha} t \rfloor}\}_{t\ge0}$ (defined as in \eqref{discretegenerator}), applied to a function $f\in C^2[0,1]$ in $x \in [0,1]$ can be written as
\begin{align}
\mathcal{\hat A}^Nf(x) \ =& \ N^{\alpha}  (1 - \frac{\eta}{N^{\alpha}})  \E \left( f\left(\frac{\sum_{i=1}^NB_i^x}{N}\right) - f(x)\right)  \label{gen13}\\ 
+& \  N^{\alpha}  \frac{\eta}{N^{\alpha}}  \sum_{k\ge 1} \p(b_{i,N}N = k) \sum_{g \ge1} \p(l_{1,N} = g)  \sum_{i =0}^k \p(\bar Y^{k, g, x} = i/k) \left( f(i/k) - f(x) \right ),
\label{gen23}
\end{align}
which can be interpreted in the same way as the generator $\mathcal{\bar A}^N$.
Again, when $\alpha=1$, part \eqref{gen13}, corresponds to the generator of a classical Wright-Fisher model and converges when $N\to \infty$ to $\frac12x(1-x)f''(x)$, which is the generator of the Wright-Fisher diffusion. When $\alpha <1$ this term becomes of order $N^{\alpha-1}$ and therefore converges to $0$. 

Part \eqref{gen23} corresponds to what happens during a bottleneck (we recall that $\mathrm{L}_\gamma/N^\alpha \to 0 $ in distribution i.e. in the new time scale the bottlenecks are instantaneous). 
It is well-known that $\{Y^{k,\lfloor kt\rfloor,x}\}_{t\ge0}$ converges in distribution, in the Skorokhod topology to the Wright-Fisher diffusion $\{Y_t\}_{t\ge0}$ with $Y_0 = x$ (see for example Chapter 2 in \cite{E}). In a similar way we can prove that $\{\bar Y^{k,\lfloor kt\rfloor,x}\}_{t\ge0}$ converges in distribution, in the Skorokhod topology to the same process. 
In fact, $\{\bar Y^{k,\lfloor kt\rfloor,x}\}_{t\ge0}$ has the distribution of the frequency process of a Wright-Fisher model, with a random initial condition. 
 This, combined with the assumptions that $b_{1,N}N\to\infty$ and $l_{1,N}/(N b_{1,N}) \to \mathrm{L}_\gamma$ in distribution, implies that
\begin{align}
\mathcal{\hat A}^Nf(x) \ {\longrightarrow}& \ \mathds{1}_{\{\alpha = 1\}}\frac12x(1-x)f''(x) 
+  \eta \int_{\R_+}  \mathrm{L}_{\gamma}(d\sigma)  \int_{[0,1]} \p(Y_{\sigma} \in dy\vert Y_0 = x) \left( f(y) - f(x) \right ).
\label{322}
\end{align}
Finally, to compute $\p(Y_{\sigma} \in dy\vert Y_0 = x) $, we use the duality relation between the Wright-Fisher diffusion and the Kingman coalescent \eqref{classicalduality}. More precisely, to compute the probability that the proportion of type $1$ individuals is $y$ at time $\sigma$, we can follow backwards in time the ancestry of the whole population. The number of ancestors is given by a Kingman coalescent started at $K_0 = \infty$. If $K_{\sigma} = j$, each one of the $j$ ancestors  is of type $1$ with probability $x$ and the fraction of the population (at time $\sigma$) that is issued from each one of the $j$ ancestors is given by a  Dirichlet distribution $D_j$.  This means that 
\begin{align}
 \eta \int_{\R_+}  \mathrm{L}_{\gamma}(d\sigma)   \p(Y_{\sigma} \in dy|Y_0=x)  &= \eta \int_{\R_+}  \mathrm{L}_{\gamma}(d\sigma) \sum_{j\ge1} \p(K_{\sigma} = j)  \int_{\Delta} D_j(\zeta) \p(\sum_{i\ge1}\zeta_i B_i^x \in dy) \label{probadual} \\
 &= \int_{\Delta} \frac{\Xi_{KS}}{(\zeta,\zeta)}(d\zeta)\p(\sum_{i\ge1}\zeta_i B_i^x \in dy), \nonumber
\end{align}
and replacing into \eqref{322}, we have that $\mathcal{\hat A}^N$ converges to $ \mathcal{\hat A}$ uniformly.
This implies that $\{ V^N_{\lfloor N^{\alpha} t \rfloor}\}_{0\le t \le T}$ converges weakly in the Skorokhod $J_1$ topology to $\{ \bar X_t\}_{0\le t \le T}$. Since convergence in $J_1$ implies convergenec in $d_\lambda$ we have the desired result.
\end{proof}

We fix $\alpha \in (0,1],  \ \eta >0$ and $\mathrm{L}_\gamma$ a probability measure on $\R_+$.
Let us consider $\{\hat N_t\}_{t\geq 0}$, the block-counting process of the subordinated Kingman coalescent characterized by $\alpha$, $\eta$ and $\mathrm{L}_\gamma$.
 As in the previous sections, we are going to prove a moment duality property between the block-counting process and the diffusion with jumps $\{\hat X_t\}_{t\geq 0}$ defined above (with the same parameters). 
\begin{theorem}
For every $x \in [0,1], \ n \in \N$, we have
$$\E(\hat X_t^n \vert \hat X_0 = x) \ = \ \E (x^{\hat N_t} \vert \hat N_0 = n).$$
\label{lastduality}
\end{theorem}
\begin{proof}
We only consider the case $\alpha < 1$, (as the extension to the case $\alpha = 1$ can be done exactly as in Section \ref{duality1}). 
Let $h(x, n)=x^n$, seen as a function of $x$.
Using  \eqref{probadual}, for every $n \in \N$, the generator $\mathcal{ \hat A}$ applied to $h$ (seen as a function of $x$) can be rewritten as
\begin{align*}
\mathcal{ \hat A}h(x,n) & = \eta\int_{\R_+}  \mathrm{L}_{\gamma}(d\sigma) \sum_{j\in \N} \p(K_{\sigma} = j)  \int_{\Delta} D_j(\zeta) \p(\sum_{i\in \N}\zeta_i B_i^x \in dy) \left (h(y,n) - h(x,n)\right ) \\
& = \mathcal{ \hat G}h(x,n),
\end{align*}
where $\mathcal{ \hat G}$ is the generator of $\{\hat N_t\}_{t\geq 0}$ applied to $h$, seen as a function of $n$.
\end{proof}

 \section*{Acknowledgements}
 The authors thank three anonymous referees for very improving commentaries.
ACG thanks Jochen Blath for helpful discussions.
AGC was supported by CONACyT Grant A1-S-14615, VMP by  DGAPA-UNAM postdoctoral program and ASJ by CONACyT Grant CB-2014/243068 and  by UNAM-PAPIIT grant IA103820.

\end{document}